\newif\ifdraft
\draftfalse

\documentclass[reqno,a4paper]{amsart}

\usepackage{mathrsfs,amsmath,amssymb,graphicx}
\usepackage{stmaryrd}
\usepackage[unicode]{hyperref}
\usepackage{latexsym}
\usepackage{layout}
\usepackage[english]{babel}
\usepackage{color}
\usepackage{subcaption}
\usepackage{fancyvrb}
\usepackage{color}
\usepackage{amsmath,amscd}
\usepackage{multicol}
\usepackage{multirow}

\usepackage{mathtools}
\usepackage{tikz} 
 
\usetikzlibrary{matrix,arrows,decorations,shapes,calc}
\tikzset{every picture/.append style={remember picture},
na/.style={baseline=-.5ex}}

\theoremstyle{plain}
\newtheorem{theorem}{Theorem}[section]
\newtheorem{lemma}[theorem]{Lemma}

\newtheorem{corollary}[theorem]{Corollary}
\newtheorem{conjecture}[theorem]{Conjecture}
\newtheorem{definition}{Definition}[section]
\theoremstyle{remark}
\newtheorem{remark}{Remark}[section]

\newcommand{\nada}[1]   {}

\newcommand{\HL}{{\rm HL}}
\newcommand{\onesum} {-{}-}

\newcommand{\Stwo} {\mathfrak S}

\newcommand{\Sbb} {\mathbb S}
\newcommand{\SSS} {\mathcal S}

\newcommand{\sphere} {\Sbb^3}

\definecolor{mygray}{rgb}{0.92,0.92,0.92}

\newcommand{\Compl}[1]{\overline{\sphere\setminus #1}}

\newcommand{\op}[1]{\operatorname{#1}}

\newcommand{\draftMMM}[1]{\ifdraft{\color{blue}#1}\fi}
\newcommand{\draftGGG}[1]{\ifdraft{\color{red}#1}\fi}
\newcommand{\draftYYY}[1]{\ifdraft{\color{orange}#1}\fi}

\graphicspath{{}}

\numberwithin{equation}{section}
\numberwithin{figure}{section}

\title{A table of $n$-component handlebody links of genus $n+1$ up to six crossings}
\author[G.\ Bellettini]{Giovanni Bellettini}
\address{Dipartimento di Ingegneria dell'Informazione e Scienze Matematiche, Universit\`a di Siena, 53100 Siena, Italy,
and International Centre for Theoretical Physics ICTP,
Mathematics Section, 34151 Trieste, Italy
}
\email{bellettini@diism.unisi.it}

\author[G.\ Paolini]{Giovanni Paolini}
\address{California Institute of Technology and Amazon Web Services, Pasadena CA, United States (work done while at University of Fribourg, Department of Mathematics, 1700 Fribourg, Switzerland)}
\email{paolini@caltech.edu}

\author[M.\ Paolini]{Maurizio Paolini}
\address{Dipartimento di Matematica e Fisica, Universit\`a Cattolica del Sacro Cuore, 25121 Brescia, Italy}
\email{maurizio.paolini@unicatt.it}

\author[Y.\ S.\ Wang]{Yi-Sheng Wang}
\address{National Center for Theoretical Sciences, Mathematics Division, Taipei 106, Taiwan}
\email{yisheng@ncts.ntu.edu.tw}

\date{\today}

\begin{document}

\thanks{}

\begin{abstract}
A handlebody link is a union of handlebodies of positive genus embedded
in $3$-space, which generalizes the notion of links in classical knot theory.
In this paper, we consider handlebody links
with    
one genus $2$ handlebody and $n-1$ solid tori, $n>1$. 
Our main result is the complete 
classification of such handlebody links with six crossings or less, up to ambient isotopy.
\end{abstract}

\maketitle

%
%
%

\section{Introduction}\label{sec:intro}
Knot tabulation has a long and rich history.
Some early work, motivated by Kelvin's
vortex theory, dates back to as early as the late 19th. 
Over the past decades,
more effort has been put into it by many 
physicists and mathematicians; all prime 
knots up to $16$ crossings are now classified \cite{HosThiWee:98}. 
In recent years
knot tabulation has been further 
generalized to other contexts. 
\cite{Mor:09a} and \cite{Mor:09b} tabulate all
prime theta curves and handcuff graphs up to seven crossings,
\cite{IshKisMorSuz:12} enumerates all irreducible
handlebody knots of genus $2$ up to six crossings, and
\cite{ChoNg:13} classifies all 
alternating Lengendrian knots 
up to seven crossings.

The aim of this paper is to 
extend the Ishii-Kishimoto-Moriuchi-Suzuki handlebody knot 
table \cite{IshKisMorSuz:12} 
to handlebody \emph{links}
with $n > 1$ components 
having total genus $n + 1$.
We call such a handlebody link
an $(n,1)$-handlebody link;
it consists of exactly one genus $2$ 
handlebody and $n-1$ solid tori.
The following theorems summarize the main results of the paper.
%
%
\nada{
In \cite{IshKisMorSuz:12} Ishii et al. recovered 
a table of all handlebody knots of genus two, up to mirror image,
having at most six crossings.
This is an important first step in the classification of  
handlebodies in the $3$-sphere $\sphere$, up to ambient isotopy, 
that generalizes the knot table
to the next simplest objects, 
where a knot is interpreted as a knotted solid torus 
in $\sphere$.


A natural further step is to consider non-connected
handlebodies having 
$n > 1$ connected components of positive genus
and total genus $g = n + 1$, the lowest possible total genus 
that does not lead to the classical knot
theory of links, considered as 
a finite number of solid 
tori linked together in $\sphere$; 
we shall call such linked 
handlebodies a $(n,1)$-handlebody link.


To obtain genus $g = n+1$ out of $n$ components we must have exactly one component, say $H$, of
genus $2$ and $n-1$ components, say $T_2, ..., T_n$, of genus $1$, i.e. $n-1$ (intertangled)
solid tori.
}
%
%
\begin{theorem}\label{teo:table}
Table \ref{tab:handlebodylinks} enumerates all   
{\it non-split}\footnote{A handlebody link $\HL$ is split if 
there is a $2$-sphere $\Stwo\subset\sphere$
with $\Stwo\cap\HL=\emptyset$ separating $\HL$ into two parts.}, 
{\it irreducible}\footnote{A handlebody link $\HL$ is reducible if there is a $2$-sphere $\Stwo$ in $\sphere$ with $\Stwo\cap\HL$ an incompressible disk in $\HL$.
}
$(n,1)$-handlebody links, up to ambient
isotopy and mirror image, by their minimal diagrams, up to six crossings. 
\end{theorem}
  
$4_1$ and $5_1$ in Table \ref{tab:handlebodylinks} 
are the only non-split, 
irreducible $(n,1)$-handlebody
links with four and five crossings, respectively.
There are $15$ handlebody links
with six crossings, among which 
$8$ have two components ($n=2$), $6$
have three components ($n=3$), and $1$ has four components ($n=4$). 
As a side note, $6_5$ in Table \ref{tab:handlebodylinks} 
also represents the famous \emph{figure eight puzzle} 
devised by Stewart Coffin \cite{BotVan:75}. Thus,
its unsplittability implies the impossibility of
solving the puzzle (Remark \ref{rem:figure_eight_puzzle}).
Also, $6_9$ in Table \ref{tab:handlebodylinks}
is an irreducible handlebody link with 
a $\partial$-irreducible complement; such phenomenon
cannot happen when $n=1$ (Remark \ref{rem:irre_HL_re_complement}). 

%
%
Our task with respect to Table \ref{tab:handlebodylinks} is two-fold.
Firstly we need to show that there is no extraneous entry, i.e. that
%
%
all entries in the table 
\begin{itemize}
\item[U.1] 
represent non-split
handlebody links,
\item[U.2] represent irreducible
handlebody links,
\item[U.3] are mutually inequivalent, up to mirror image, 
\item[U.4] attain minimal crossing numbers.
\end{itemize}
Secondly we have to prove that the table is \emph{complete}; namely,
there is no missing handlebody link with $6$ crossings or less.
   
In Section \ref{sec:uniqueness} we prove U.1-U.3,
making use of invariants such as the linking number \cite{Miz:13},
irreducibility criteria \cite{BePaWa:20}, 
and the Kitano-Suzuki invariant \cite{KitSuz:12} (Theorems 
\ref{teo:unsplittability}, \ref{teo:irreducibility}, and 
\ref{teo:uniqueness}, respectively).  
\nada{
U.4 is a direct consequence of U.1-U.3 and E.1-E.2 below at least for a lower number of crossings so that
it is automatically satisfied assuming we prove U.1-U.4 and E.1-E.2 
for increasing values of the crossing number.
}
We prove the completeness of Table \ref{tab:handlebodylinks} 
by exhausting all---except for those obviously 
non-minimal---diagrams of non-split, irreducible $(n,1)$-handlebody links up to six crossings (Section \ref{sec:completeness}).

We first observe that the underlying plane graph of 
a diagram of a non-split, irreducible $(n,1)$-handlebody link
necessarily has edge connectivity equal to $2$ or $3$;
for the sake of simplicity, such a diagram is said to 
have $2$- or $3$-connectivity, respectively. 
Diagrams with $3$-connectivity up to six crossings
are generated by a computer code, 
whereas to recover handlebody links represented by diagrams with 
$2$-connectivity, we employ 
the knot sum---the \emph{order-$2$ vertex connected sum}---of spatial 
graphs \cite{Mor:09a}. In more detail,
a minimal diagram $D$ with $2$-connectivity can be  
decomposed by decomposing 
circles\footnote{a circle that intersects $D$ at two different arcs.} 
into simpler tangle diagrams, each of which induces a spatial 
graph that admits a minimal diagram with $3$- or $4$-connectivity, 
as illustrated in Fig.\ \ref{fig:decomposition_type_2}. 
%
%
This decomposition allows us to recover the handlebody link represented 
by $D$ by performing the \emph{knot sum}  
between prime links and a \emph{spatial graph} 
that admits a minimal diagram with $3$-connectivity. 
\begin{figure}[ht]
\def\svgwidth{0.9\columnwidth}
\begingroup%
  \makeatletter%
  \providecommand\color[2][]{%
    \errmessage{(Inkscape) Color is used for the text in Inkscape, but the package 'color.sty' is not loaded}%
    \renewcommand\color[2][]{}%
  }%
  \providecommand\transparent[1]{%
    \errmessage{(Inkscape) Transparency is used (non-zero) for the text in Inkscape, but the package 'transparent.sty' is not loaded}%
    \renewcommand\transparent[1]{}%
  }%
  \providecommand\rotatebox[2]{#2}%
  \newcommand*\fsize{\dimexpr\f@size pt\relax}%
  \newcommand*\lineheight[1]{\fontsize{\fsize}{#1\fsize}\selectfont}%
  \ifx\svgwidth\undefined%
    \setlength{\unitlength}{3685.03937008bp}%
    \ifx\svgscale\undefined%
      \relax%
    \else%
      \setlength{\unitlength}{\unitlength * \real{\svgscale}}%
    \fi%
  \else%
    \setlength{\unitlength}{\svgwidth}%
  \fi%
  \global\let\svgwidth\undefined%
  \global\let\svgscale\undefined%
  \makeatother%
  \begin{picture}(1,0.24615385)%
    \lineheight{1}%
    \setlength\tabcolsep{0pt}%
    \put(0,0){\includegraphics[width=\unitlength,page=1]{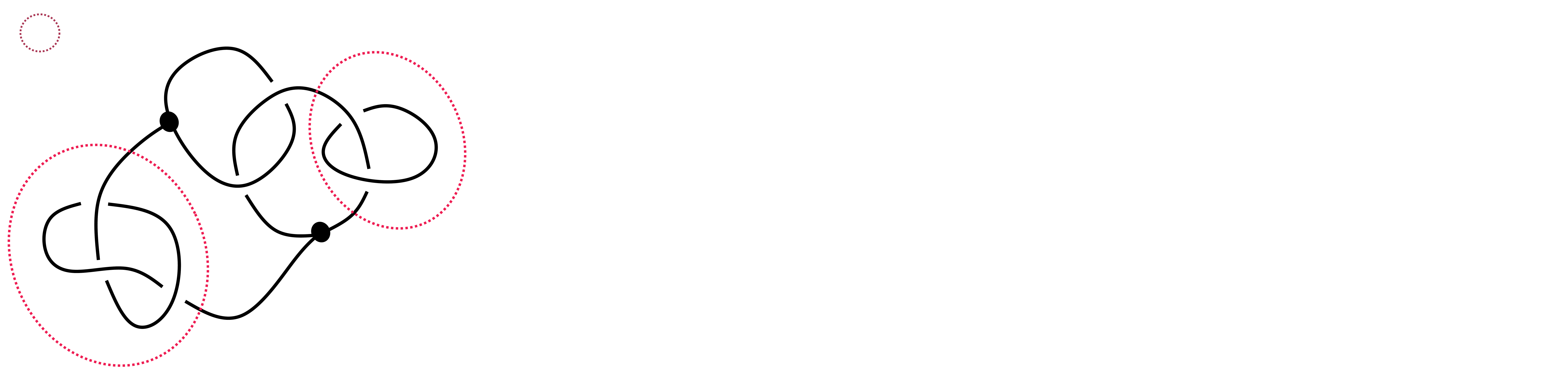}}%
    \put(0.0610321,0.22398191){\color[rgb]{0,0,0}\makebox(0,0)[lt]{\lineheight{1.25}\smash{\begin{tabular}[t]{l}{\footnotesize decomposing circle}\end{tabular}}}}%
    \put(0,0){\includegraphics[width=\unitlength,page=2]{decomposing_type_2_diagram.pdf}}%
    \put(0.70746202,0.01089694){\color[rgb]{0,0,0}\makebox(0,0)[lt]{\lineheight{1.25}\smash{\begin{tabular}[t]{l}{\footnotesize knot sum of spatial graphs}\end{tabular}}}}%
    \put(0,0){\includegraphics[width=\unitlength,page=3]{decomposing_type_2_diagram.pdf}}%
    \put(0.32583367,0.00734528){\color[rgb]{0,0,0}\makebox(0,0)[lt]{\lineheight{1.25}\smash{\begin{tabular}[t]{l}{\footnotesize induced spatial graphs}\end{tabular}}}}%
    \put(0,0){\includegraphics[width=\unitlength,page=4]{decomposing_type_2_diagram.pdf}}%
  \end{picture}%
\endgroup%
   
\caption{Decomposing a minimal diagram with $2$-connectivity.}
\label{fig:decomposition_type_2}
\end{figure}

Once a list containing all possible minimal diagrams
of non-split, irreducible handlebody links
is produced, we examine 
each entry on the list manually (Appendix \ref{sec:code}), and
show that either it is non-minimal or it represents a 
handlebody link ambient isotopic to one in Table \ref{tab:handlebodylinks},
up to mirror image. This proves the completeness, and also implies U.4, given U.1-U.3.

\nada{
The order-$2$ vertex connected sum of spatial graphs is 
\emph{the knot sum of spatial graphs}, defined via
trivial ball-arc pairs \cite{Mor:09a}. 
It is well-defined 
once a choice of an arc in the graph, 
as well as an orientation of the arc, are specified; 
in our case, only two trivalent vertices are involved, so for a 
$n$-component spatial
graph, there are $2(n+2)$ such choices. 
Alternatively, one could think of the order-$2$ vertex connected sum 
between a spatial graph and a link
as a finite-set-valued operation. 
}

\begin{theorem}\label{teo:chiral_hl}
All but $5_1$, $6_3$, $6_6$, $6_7$, $6_8$, 
$6_{10}$ in
Table \ref{tab:handlebodylinks} are achiral. 
\end{theorem}
The main tool used to inspect chirality 
is Theorem \ref{teo:uniqueness_decomposition_special_case},
where we prove a uniqueness result for the decomposition of
non-split, irreducible handlebody links 
in terms of order-$2$ connected sum of handlebody-link-disk pairs 
(Definition \ref{def:two_sum}).
\begin{theorem}\label{teo:reducible}
Table \ref{tab:reducible} enumerates all non-split, reducible 
$(n,1)$-handlebody links
up to $6$ crossings, up to mirror image.
\end{theorem}
Theorem \ref{teo:reducible} follows from 
the irreducibility of handlebody links 
in Table \ref{tab:handlebodylinks} 
and a uniqueness factorization theorem 
(Theorem \ref{teo:uniqueness_onesum}) 
for non-split, reducible $(n,1)$-handlebody links
in terms of order-$1$ connected sum (Definition \ref{def:one_sum}).

The structure of the paper is the following.
Basic properties of handlebody links
are reviewed in Section \ref{sec:prelim}; 
uniqueness, unsplittability, and irreducibility of handlebody links 
in Table \ref{tab:handlebodylinks}
are examined in Section \ref{sec:uniqueness}. 
The completeness of the table is discussed in 
Section \ref{sec:completeness}. 
Section \ref{sec:chirality} introduces the notion of 
decomposable handlebody links, which is used for examining  
the chirality of handlebody links in Table \ref{tab:handlebodylinks}.
Lastly, a complete classification 
of non-split, reducible handlebody links, up to $6$ crossings,
is given in Section \ref{sec:reducible}. In the appendix 
we include an analysis on the output of the code, 
available at \url{http://dmf.unicatt.it/paolini/handlebodylinks/}.

Throughout the paper we work in the $\op{PL}$ category;
for the illustrative purposes, the drawings often appear smooth.
In the case of $3$-dimensional submanifolds in $\sphere$,
the $\op{PL}$ category is equivalent to the smooth category due to
\cite[Theorem $5$]{Bro:62}, \cite[Theorems $7.1$, $7.4$]{HirMaz:74},
\cite[Theorem $8.8$, $9.6$, $10.9$]{Mun:66}.



%
%
%
%
\begin{table}[h]
\caption{Non-split, irreducible handlebody links up to six crossings.}
\label{tab:handlebodylinks}
\includegraphics[height=0.18\textwidth]{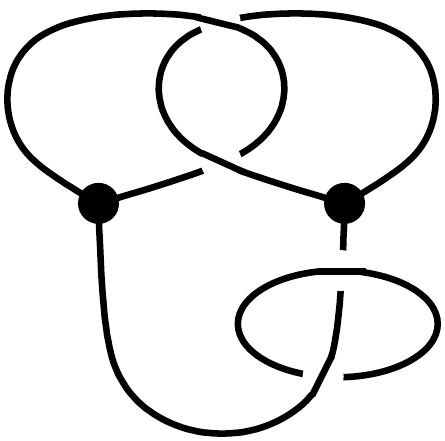}$4_1$~~~~
\includegraphics[height=0.18\textwidth]{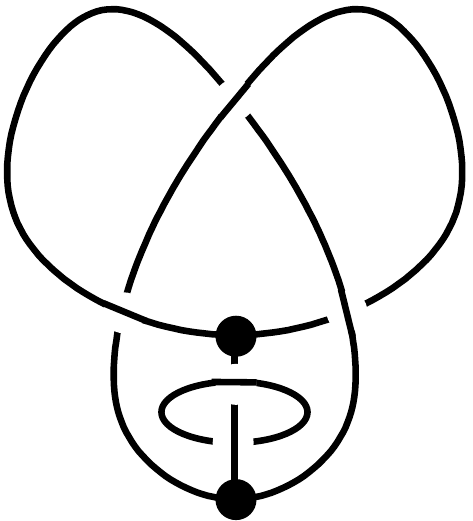}$5_1$
\\
\smallskip
\includegraphics[height=0.18\textwidth]{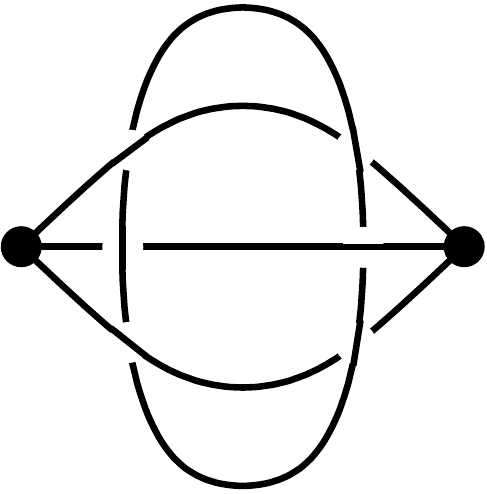}$6_1$
\includegraphics[height=0.18\textwidth]{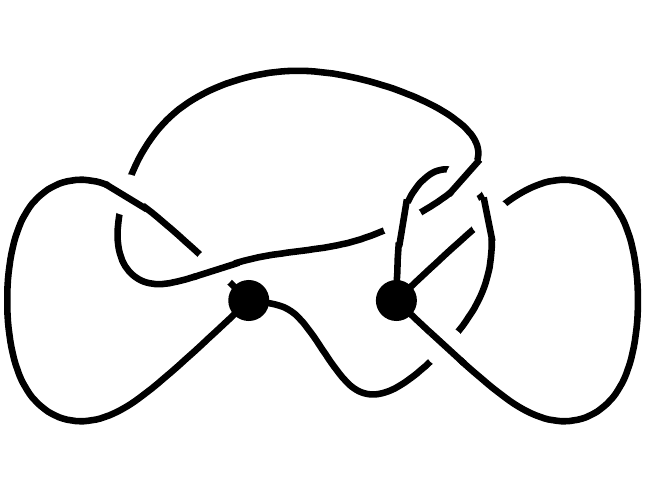}$6_2$
\includegraphics[height=0.18\textwidth]{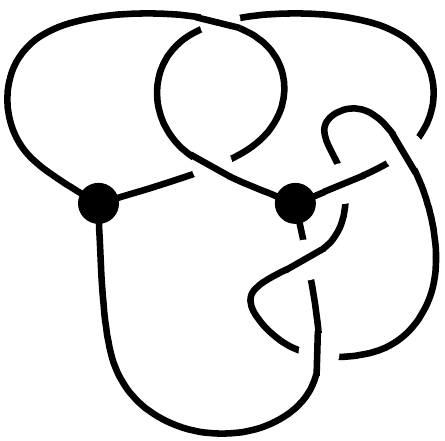}$6_3$
\includegraphics[height=0.18\textwidth]{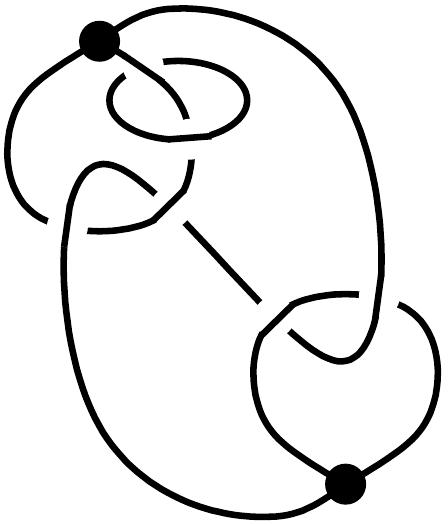}$6_4$
\\
\smallskip
\includegraphics[height=0.18\textwidth]{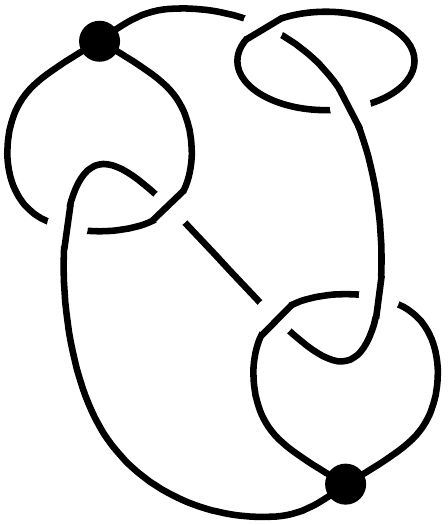}$6_5$
\includegraphics[height=0.18\textwidth]{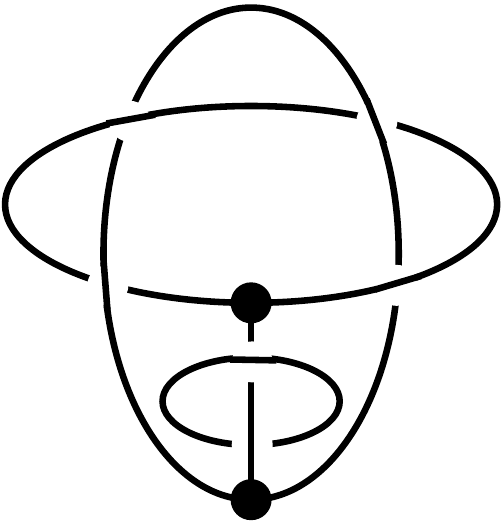}$6_6$
\includegraphics[height=0.18\textwidth]{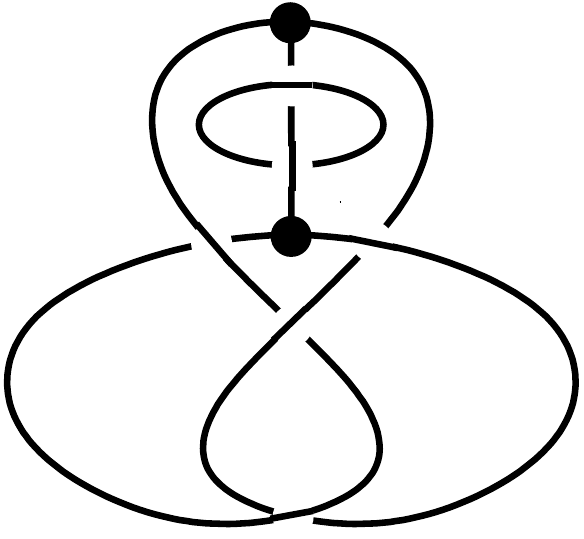}$6_7$
\includegraphics[height=0.18\textwidth]{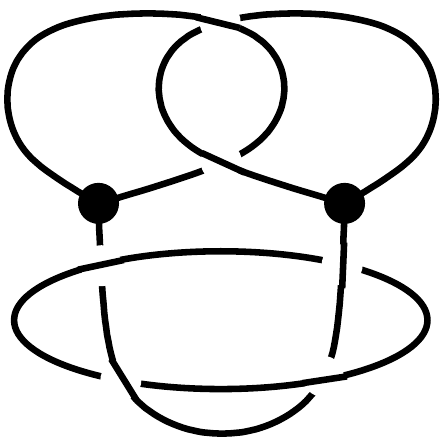}$6_8$
\\
\smallskip
\includegraphics[height=0.18\textwidth]{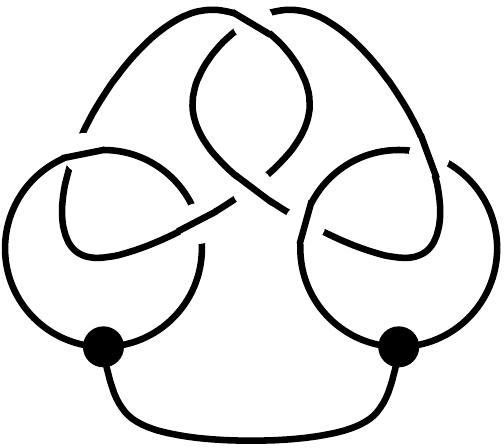}$6_9$
\includegraphics[height=0.18\textwidth]{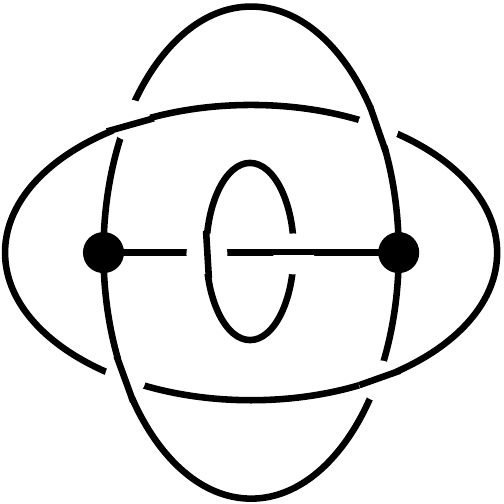}$6_{10}$
\includegraphics[height=0.18\textwidth]{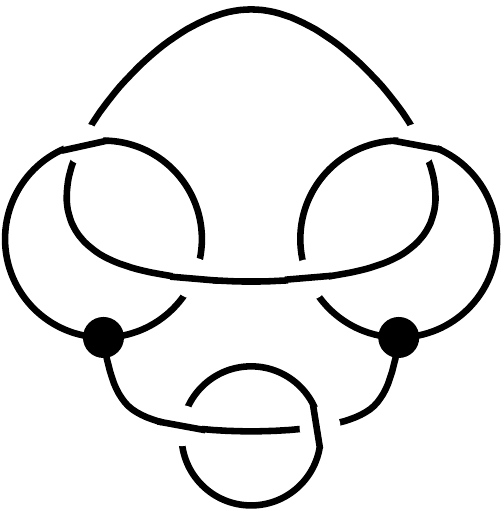}$6_{11}$
\includegraphics[height=0.18\textwidth]{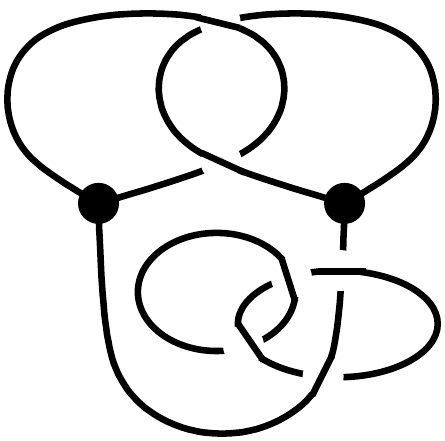}$6_{12}$
\\
\smallskip
\includegraphics[height=0.18\textwidth]{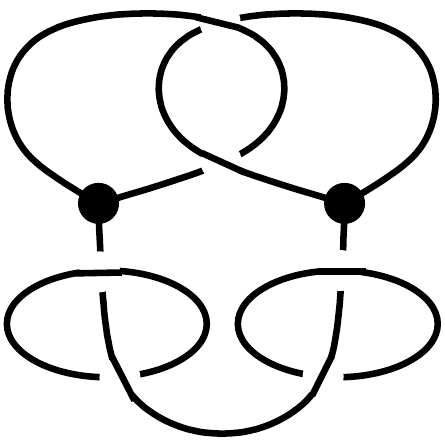}$6_{13}$
\includegraphics[height=0.18\textwidth]{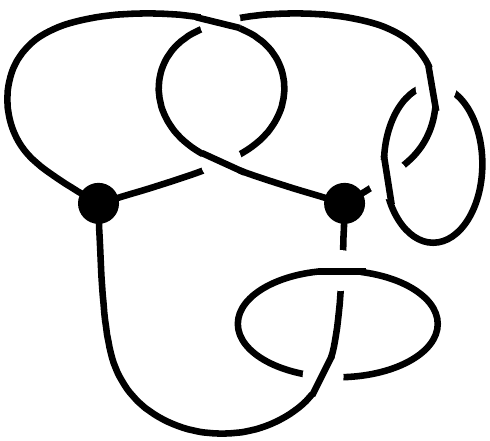}$6_{14}$
\includegraphics[height=0.18\textwidth]{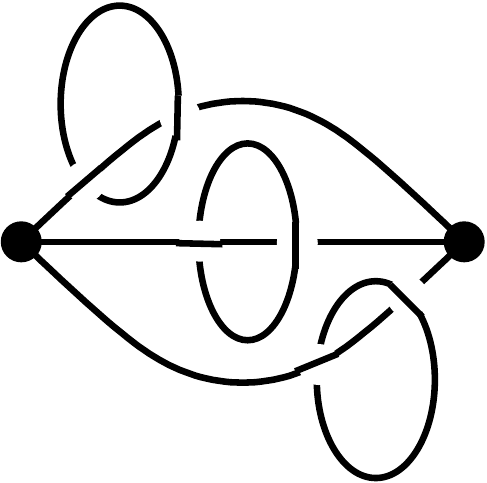}$6_{15}$
\end{table}

%
%


\section{Preliminaries}\label{sec:prelim}

\subsection{Handlebody links and spatial graphs}
\begin{definition}[\textbf{Embeddings in $\sphere$}]
A handlebody link $\HL$ (resp.\ a spatial graph $G$) is
an embedding of finitely many handlebodies of positive genus 
(resp.\ a finite graph\footnote{A 
finite graph is a graph with finitely many
vertices and edges; in addition, we require that no component  
has a positive Euler characteristic, to avoid trivial objects. 
A circle is regarded
as a graph without vertices as in \cite{Ish:08}.}) in the oriented $3$-sphere $\sphere$.
\end{definition}
The \textit{genus} of a handlebody link is the sum of
the genera of its components;
a spatial graph is \textit{trivalent} if the underlying graph is
trivalent (each node has degree 3). By a slight abuse of notation, we also use
$\HL$ (resp.\ $G$) to denote the image of the embedding
in $\sphere$. Let $r\HL$ (resp.\ $rG$) be the mirror image
of $\HL$ (resp.\ $G$).
\begin{definition}[\textbf{Equivalence}]
Two handlebody links $\HL$, $\HL'$ (resp.\ spatial gra\-phs $G,G'$) are  
equivalent if they are ambient isotopic;
they are equivalent up to mirror image 
if $\HL$ (resp. $G$) is equivalent to $\HL'$ or $r\HL'$ (resp.\ $G'$ or $rG'$).
\end{definition}
%
\nada{
\draftGGG{was not Cerf?} 
\draftYYY{Oh..yes. Cerf proves that $O(4)\hookrightarrow Diff(S^3)$
induces an isomorphism on $\pi_0$ and Hatcher proves the inclusion
is a homotopy equivalence (The Smale conjecture); so, for our purpose,
Cert's theorem is enough. Maybe we cite both papers!? Since 
we are working in PL, we perhaps need to mention PL=Smooth.}
\draftYYY{I found the PL version of the theorem...}
}
A regular neighborhood of a spatial graph 
defines a handlebody link,  
up to equivalence \cite[$3.24$]{RouSan:82}, and
a spine of a handlebody link $\HL$
is a spatial graph $G$ with $\HL$ a regular neighborhood of $G$ 
\cite{Ish:08}.
In practice, it is more convenient to consider
spines that are trivalent.
\begin{lemma}
Every handlebody link admits a (trivalent) spine.
\end{lemma}
\begin{proof}
It suffices to prove the connected case. We suppose 
$\op{HK}$ is a handlebody knot of genus $g$, and
$\mathbf{D}=\{D_1,\cdots,D_{3g-3}\}$ 
is a set of disjoint
incompressible disks in $\op{HK}$ 
such that 
the complement   
$\op{HK}\setminus \bigcup_i N(D_i)$
of their tubular neighborhoods $N(D_i)$ in $\op{HK}$
consists of $2(g-1)$ $3$-balls, 
each of which has non-trivial intersection with
exactly three components of 
$\coprod_{i=1}^{3g-3} \partial \overline{N(D_i)}$.
Such a disk system always exists.

Let disks $D_{i1},D_{i2}, D_{i3}$ be 
components of $B_i\cap \big(\bigcup_{k=1}^{3g-3} \overline{N(D_k)}\big)$,
and choose points $v_{i1},v_{i2},v_{i3}$ in the interior of 
$D_{i1},D_{i2}, D_{i3}$ and a point $v_i$ in
the interior of $B_i$. Then, join $v_i$ to $v_{ij}$ by a path for each $j$; this gives us a trivalent vertex. 
Repeat the construction for every $i$, and then glue the $v_{ij}$ 
together so that the
vertices $v_{ij}$ and $v_{i'j'}$ are identified if 
they are in the same $\overline{N(D_k)}$, for some $k$.
This way, we obtain a connected trivalent spine of $\op{HK}$
with $2(g-1)$ trivalent vertices. 
\end{proof}
In general, a trivalent spine of 
a $n$-component handlebody link
of genus $g$ has $2(g-n) = 2t$  
trivalent vertices, and 
we call such a handlebody link 
a $(n,t)$-handlebody link. This paper is primarily 
concerned with the case $t=1$.

\nada{
\begin{definition}[\textbf{Diagram}]
Given a plane graph $G$ with trivalent and quadrivalent vertices,
an associated diagram $D$ is the plane graph
$G$ with each quadrivalent vertex replaced by a crossing.
\end{definition}
Since, for each quadrivalent vertex, 
there are two possible crossings, 
if $G$ has $c$ quadrivalent then there 
are $2^{c-1}$ different associated diagrams, up to mirror image.  
}

\nada{  
\draftMMM{Actually, we should allow for two-valent $0$-simplices; or am I
missing something? On the other hand we prevent the presence of one-valent
or zero-valent $0$-simplices
}
\draftYYY{Yes. It is just some $1$-simplicies attached along
their end points, so it
allows two-valent $0$-simplicies, and $1$-valent 
$0$-simplicies (but would that constitutes a problem for our arguments);
on the other hand, since it is $1$-dimension, disconnected $0$-simplicies
should be excluded.}
}

\subsection{Diagrams}   
Let $\mathbb S^k=\mathbb R^k \cup \infty$, and without loss of generality,
it may be assumed handlebody links or spatial graphs are away from $\infty$.
\begin{definition}[\textbf{Regular projection}]
A regular projection of a spatial graph $G$
is a projection $\pi:\sphere\setminus\infty\rightarrow 
\mathbb S^2\setminus\infty$
such that the set $\pi^{-1}(x) \cap G$ is finite
with its cardinality $\#(\pi^{-1}(x) \cap G)\leq 2$ 
for any $x\in \mathbb S^2\setminus\infty$,
and no $0$-simplex of the polygonal subset $G$ of $\sphere$  
is in the preimage of a double point, a 
double point being a point $x \in \mathbb S^2\setminus\infty$
with $\#(\pi^{-1}(x) \cap G) = 2$. 
\end{definition}
 
As with the case of knots, 
up to ambient isotopy, 
every spatial graph admits a regular projection:
the idea is to choose a vector $v$ neither parallel
to a $1$-simplex in the polygonal subset $G\subset \sphere\setminus \infty=\mathbb{R}^3$  
nor in a plane containing a $0$-simplex and a $1$-simplex
or two $1$-simplices; then isotopy $G$ slightly to
remove those points $x$ with $\#\pi_v^{-1}(x)\cap G>2$,
where $\pi_v$ is the projection onto the plane normal to $v$.   


\begin{definition}[\textbf{Diagram of a spatial graph}]
A diagram of a spatial graph $G$ is
the image of a regular projection of $G$ 
with relative height information added to each double point.
\end{definition}

The convention is to make breaks in the line corresponding to 
the strand passing underneath; thus each double point becomes
a \textit{crossing} of the diagram.

\begin{definition}[\textbf{Diagram of a handlebody link}]
A diagram of a handlebody link $\HL$
is a diagram of a spine of $\HL$.
\end{definition}

A diagram of $G$ (resp.\ $\HL$) is trivalent if it is
obtained from a regular projection of a trivalent spatial graph
(resp.\ spine). 
 
\begin{definition}[\textbf{Crossing numbers}]
The crossing number $c(D)$ of a 
diagram $D$ of a handlebody link $\HL$ (resp.\ of a spatial graph $G$)
is the number of crossings in $D$. 
The crossing number $c(\HL)$   
of $\HL$ (resp.\ $c(G)$ of $G$) is
the minimum of the set
\[
\{ c(D)\mid \text{$D$ a diagram of $\HL$ $($resp. $G)$}\}.
\] 
\end{definition}

\begin{definition}[\textbf{Minimal diagram}]
A minimal diagram $D$ of a handlebody link $\HL$ (resp.\ 
of a spatial graph $G$) 
is a diagram of $\HL$ (resp. $G$) with $c(D) = 
c(\HL)$ (resp. $c(D)=c(G)$).
\end{definition}
Every multi-valent vertex in a minimal diagram $D$ 
can be replaced with some trivalent vertices by the inverse of
the contraction move \cite[Fig.\ $1$]{Ish:08} without changing
the crossing number, so
for a handlebody link (resp.\ a spatial graph)
there always exists a trivalent minimal diagram $D$. 
From now on, we shall use the term ``a diagram"
to refer to a trivalent diagram of 
either a spatial graph or a handlebody link.

Observe that, regarding each crossing as a quadrivalent
vertex, we obtain a plane graph, a finite graph
embedded in the $2$-sphere.
In practice, we work backward and start with 
a plane graph having only trivalent and quadrivalent vertices,
and produce diagrams 
by replacing quadrivalent vertices with under- or over-crossings. 
If the plane graph has $2t$ trivalent vertices 
and $c$ quadrivalent vertices, then we can recover $2^{c-1}$ diagrams from it, up to mirror image. 
In particular, a $c$-crossing $(n,t)$-handlebody link
can be recovered from one of these plane graphs.
Therefore if one can enumerate all plane graphs
with $2t$ trivalent vertices and up to $c$ quadrivalent vertices,
then one can find all $(n,t)$-handlebody links
up to $c$ crossings.

\subsection{Moves}

\begin{definition}[\textbf{Moves}]
Local changes in a diagram depicted in Fig.\ \ref{fig:reidemeister_a}
and Fig.\ \ref{fig:reidemeister_b} are called 
generalized Reidemeister moves, 
and the local change in Fig.\ \ref{fig:IH-move} 
is called an IH-move.
\end{definition}

Note that spines of equivalent handlebody links
might be inequivalent as spatial graphs;
indeed, the following holds.

\begin{theorem}[\text{\cite[Theorem $2.1$]{Kau:89}, \cite{Yet:89}}]\label{thm:G_Reid_moves_spatial_graph}
Two trivalent spatial graphs
are equivalent if and only if their diagrams
are related by a finite sequence of generalized Reidemeister moves.
\end{theorem}

\begin{theorem}[\text{\cite[Corollary $2$]{Ish:08}}]\label{thm:IH_move_handlebody_link}
Two handlebody links  
are equivalent if and only if their trivalent diagrams
are related by a finite sequence of 
generalized Reidemeister moves and IH-moves.
\end{theorem}
\begin{figure}[ht]
\includegraphics[height=0.18\textwidth]{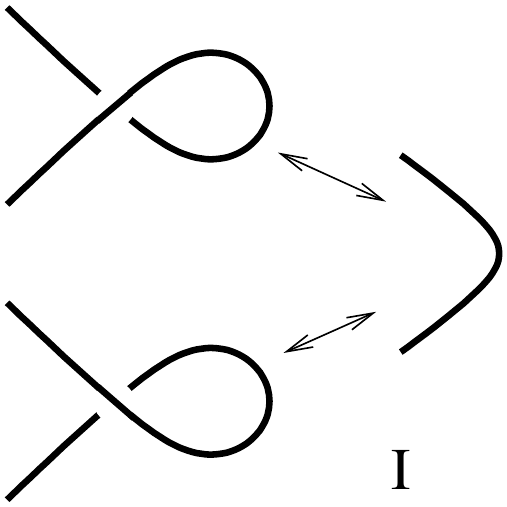}~~
\includegraphics[height=0.18\textwidth]{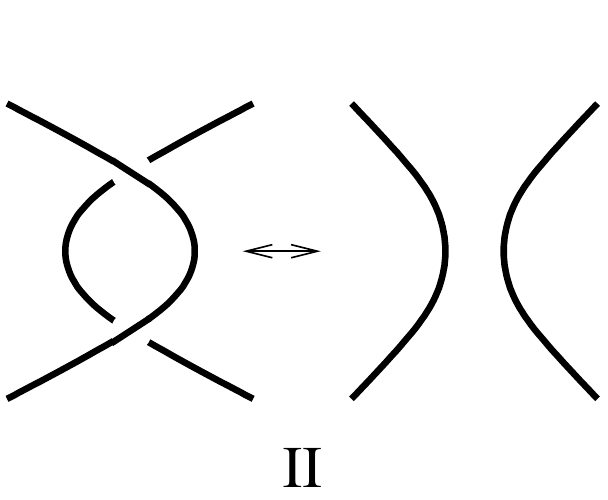}~~
\includegraphics[height=0.18\textwidth]{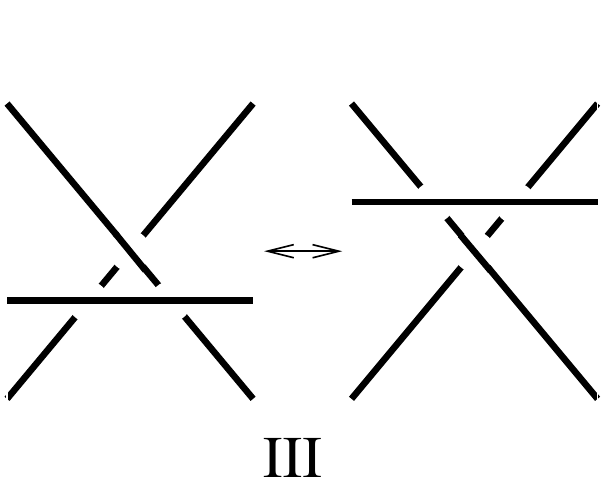}
\caption{Classical Reidemeister moves of type I, II, III.}
\label{fig:reidemeister_a}
\end{figure}
\begin{figure}[ht]
\includegraphics[height=0.18\textwidth]{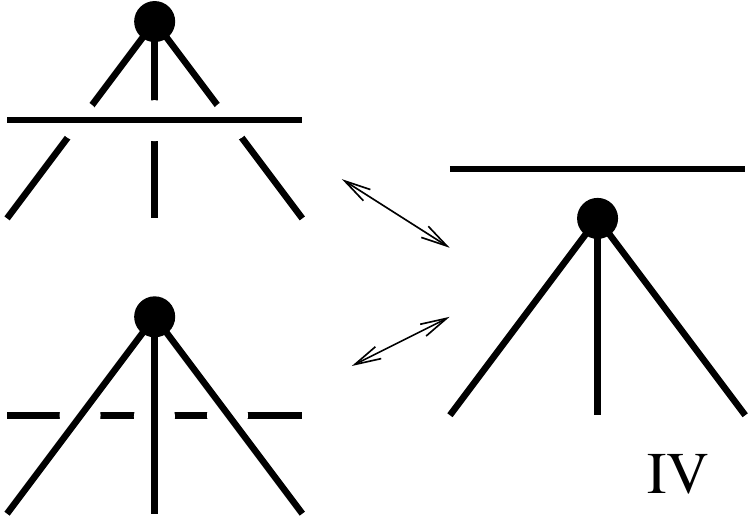}~~~~
\includegraphics[height=0.18\textwidth]{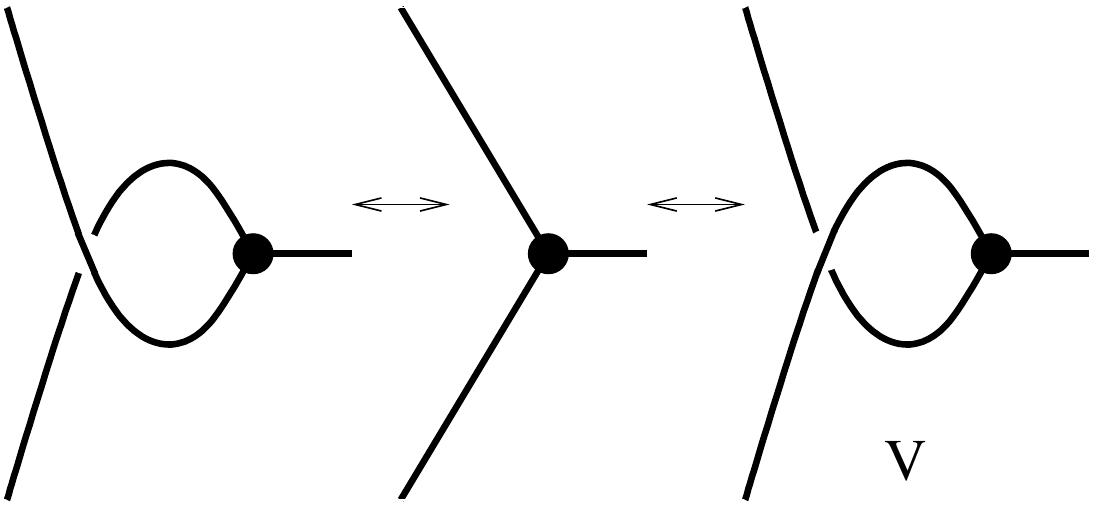}
\caption{Reidemeister moves IV and V involve a triple point.}
\label{fig:reidemeister_b}
\end{figure}
\begin{figure}[ht]
\includegraphics[height=0.18\textwidth]{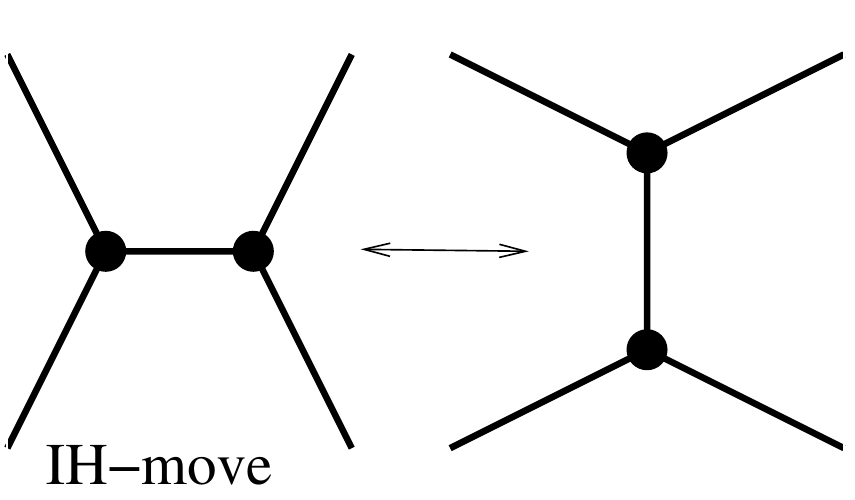}
\caption{The IH-move.}
\label{fig:IH-move}
\end{figure}

When analyzing the data from the code (Appendix \ref{sec:code}), 
it is more convenient to call a diagram IH-minimal
if the number of crossings cannot be reduced by 
generalized Reidemeister moves and IH moves, that is,
``minimal'' as a diagram of a \emph{handlebody link}, and 
to call a diagram R-minimal   
if the number of crossings 
cannot be reduced by generalized Reidemeister moves,
that is, ``minimal'' as a diagram of a \emph{spatial graph}.

\subsection{Non-split, irreducible handlebody links}
\begin{definition}[\textbf{Edge connectivity of a graph}]
The edge-connectivity of a graph is 
the minimum number of edges whose deletion disconnects
the graph.
\end{definition}
%

\begin{definition}[\textbf{Connectivity of a diagram}]
A diagram has $e$-connectivity if its underlying plane
graph has edge-connectivity $e$.
\end{definition}
\begin{definition}[\textbf{Split handlebody link}]
A handlebody link $\HL$ is split if there exists a $2$-sphere
$\Stwo\subset \sphere$ such that $\Stwo\cap \HL=\emptyset$ and
both components of the complement $\Compl{\Stwo}$ have non-trivial intersection with $\HL$.
\end{definition}  
\begin{definition}[\textbf{Reducible handlebody link}]\label{def:reducible}
A handlebody link $\HL$ is re\-du\-ci\-ble if
its complement admits a $2$-sphere $\Stwo$ such that 
$\Stwo\cap \HL$ is an incompressible disk in $\HL$;
otherwise it is irreducible.
\end{definition}
Note that $\Stwo$ in Definition \ref{def:reducible} 
factorizes $\HL$ into two handlebody links,
each called a \emph{factor} of the factorization of $\HL$.

A diagram with $0$-connectivity (resp.\ $1$-connectivity)
represents a split (resp.\ reducible) handlebody link,  
so only diagrams with connectivity
greater than $1$ are of interest to us; on the other hand, the connectivity of 
a diagram of a $(n,t)$-handlebody link with $t>0$ 
cannot exceed $3$. 

\nada{
\begin{definition}[\textbf{Type-2 \& Type-3}]\label{def:type-two-three}
A non-split, irreducible handlebody link is of \emph{type-2} 
if it admits a minimal
diagram with edge connectivity two; 
otherwise it is of \emph{type-$3$}.
\end{definition}
It follows that any minimal diagram $\mathcal{D}$ of a type-$3$ handlebody link has edge connectivity $e(\mathcal{D}) = 3$ due to the presence of trivalent nodes.
To understand the relation between type-$2$ and type-$3$,
}
Now, we recall the order-$2$ vertex connected sum 
between spatial graphs \cite{Mor:09a}, which is 
used to produce handlebody links represented by 
minimal diagrams with $2$-connectivity. 
A trivial ball-arc pair of a spatial graph $G$
is a $3$-ball $B$ with $G\cap B$ a trivial tangle in $B$; 
it is oriented if an orientation of $G\cap B$ is given.
\begin{definition}[\textbf{Knot sum}] 
Given two spatial graphs $G_1,G_2$ with oriented 
trivial ball-arc pairs $B_1,B_2$ of $G_1,G_2$, respectively,
their order-$2$ vertex connected sum $(G_1,B_1)\#(G_2,B_2)$ 
is a spatial graph obtained by removing the interiors of
$B_1,B_2$ and gluing the resulting manifolds
$\Compl{B_1}$ and $\Compl{B_2}$ by 
an orientation-reserving homoemorphism
\[h:\big(\partial (\Compl{B_1}\big),\partial(G_1\cap B_1))
\rightarrow (\partial \big(\Compl{B_2}\big),\partial (G_2\cap B_2)).\]
The notation $G_1\#G_2$ denotes the set
of order-$2$ vertex connected sums
of $G_1,G_2$ given by all possible trivial ball-arc pairs.
\end{definition} 
Since an order-$2$ vertex connected sum depends only
on the edges of $G_1,G_2$ intersecting with $B_1,B_2$
and their orientations, $G_1\# G_2$ 
is a finite set.
\section{Uniqueness, non-splittability, and irreducibility}\label{sec:uniqueness}
Recall that, given a finite group $\mathsf{G}$,  
the Kitano-Suzuki invariant $ks_{\mathsf{G}}(\HL)$ of 
a handlebody link $\HL$ is the number of 
conjugate classes of homomorphisms from
$\pi_1(\Compl{\HL})$ to $\mathsf{G}$ \cite{KitSuz:12}.
Table \ref{tab:uniqueness} lists 
the invariants $ks_{\mathsf{A}_4}(\HL)$ and $ks_{\mathsf{A}_5}(\HL)$
of each handlebody link $\HL$ in Table \ref{tab:handlebodylinks},
$\mathsf{A}_k$ being the alternating group on $k$ letters,
as well as an upper bound
of the rank of $\pi_1(\Compl{\HL})$
computed by Appcontour \cite{appcontour}.

The entry ``split'' refers to 
the split handlebody link $\op{HL}$ 
given by a trivial handlebody knot 
and an unknotted solid torus;
the entry ``fake $6_5$'' is the split handlebody link consisting of the handlebody knot 
$\op{HK}4_1$, 
Ishii-Kishimoto-Suzuki-Moriuchi's $4_1$ in \cite{IshKisMorSuz:12}, and an unknotted solid torus;
%
the entry ``fake $6_{11}$'' is $6_{11}$ in Table \ref{tab:handlebodylinks} with one of the bottom crossings reversed, thus making the lower solid torus component split off.

\begin{table}[h!]
  \begin{center}
    \caption{Kitano-Suzuki invariant for entries in Table \ref{tab:handlebodylinks}.
}
    \label{tab:uniqueness}
    \begin{tabular}{c|c|r|r|r|} 
      \textbf{handlebody link} & \textbf{components} & $ks_{\mathsf{A}_4}$ & $ks_{\mathsf{A}_5}$& rank\\
      \hline
      split & trivial + unknot & 178 & 3675   & 3\\
      $4_1$ & trivial + unknot & 114 & 600    & 3\\
      $5_1$ & trivial + unknot & 98 & 660     & $\leq$ 4\\
      $6_1$ & trivial + unknot & 90 & 600     & 3\\
      $6_2$ & trivial + unknot & 106 & 689    & 3\\
      $6_3$ & trivial + unknot & 90 & 469     & 3\\
      $6_4$ & HK$4_1$ + unknot & 106 & 689    & 3\\
      $6_5$ & HK$4_1$ + unknot & 210 &        & $\leq$ 4\\
      fake $6_5$ & HK$4_1$ + unknot & 274 &   &\\
      $6_6$ & trivial + unknot & 130 & 1380   & 3\\
      $6_7$ & trivial + unknot & 98 & 597     & $\leq$ 4\\
      $6_8$ & trivial + unknot  & 114 & 1401  & 3\\
      $6_9$ & trivial + 2 unknots & 310 & 1841    & 4\\
      $6_{10}$ & trivial + 2 unknots & 326 &      & 4\\
      $6_{11}$ & trivial + 2 unknots & 486 & 5876 & 4\\
      fake $6_{11}$ & trivial + 2 unknots & 694 &  &\\
      $6_{12}$ & trivial + 2 unknots  & 502 & 5883 & 4\\
      $6_{13}$ & trivial + 2 unknots  & 822 &      & 4\\
      $6_{14}$ & trivial + 2 unknots  & 486 & 5876 & 4\\
      $6_{15}$ & trivial + 3 unknots  & 1242 &     & 5\\
      \hline
    \end{tabular}
  \end{center}
\end{table}

\begin{theorem}[\textbf{Uniqueness}]\label{teo:uniqueness}
Entries in Table \ref{tab:handlebodylinks} are all inequivalent.
\end{theorem}

\begin{proof}
All entries in Table \ref{tab:handlebodylinks} except for 
the pairs $(6_2,6_4)$ and $(6_{11},6_{14})$
are distinguished by comparing their 
$ks_{\mathsf{A}_4}$ and $ks_{\mathsf{A}_5}$ invariants 
(shown in Table \ref{tab:uniqueness}).
%
On the other hand, $6_2$ and $6_4$ cannot be equivalent 
because the removal of the ``unknot'' component produces inequivalent
handlebody knots: one being trivial, the other being $\op{HK}4_1$.
Similarly, one can distinguish $6_{11}$ and $6_{14}$ 
by removing the solid torus component having a non-trivial linking number
with the genus $2$ handlebody component \cite{Miz:13} in
each of them, and observing that, for $6_{14}$, 
the resulting handlebody link 
is $4_1$, whereas for $6_{11}$, we get 
the trivial split handlebody link.
%
\end{proof}

\begin{remark}\label{rem:pi_1_cannot_distinguish}
The pairs $(6_2, 6_4)$ and $(6_{11}, 6_{14})$ in fact 
have homeomorphic complements, and 
hence the fundamental group cannot discriminate.
Fig.\ \ref{fig:links_6_2_6_4} and \ref{fig:links_6_11_6_14}
illustrate how to obtain the complements of $6_2$ and $6_{11}$ 
from $6_4$ and\ $6_{14}$, respectively, 
via $3$-dimensional Dehn-twists (indicated by arrows).
\end{remark}
\begin{remark}\label{rem:figure_eight_puzzle}
$6_5$ viewed as a diagram of a spatial graph 
is the notorious figure eight puzzle devised by Steward Coffin \cite{BotVan:75}. The goal of the puzzle is to free the circle 
component from the knotted handcuff graph, i.e.\ to obtain the fake $6_5$ as a spatial graph. 
The impossibility of solving the puzzle then follows from
computing $ks_{\mathsf{A}_4}(\bullet)$ of $6_5$ and fake $6_5$ (Table \ref{tab:uniqueness}). See \cite{Ber:03}, \cite{Mel:06} 
for other proofs of this.
\end{remark}
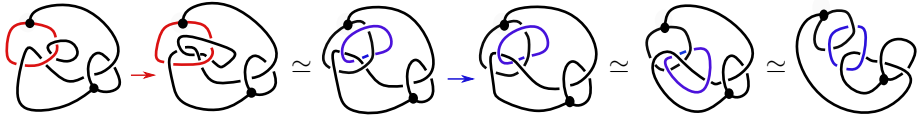
\begin{figure}[ht]
\def\svgwidth{0.95\columnwidth}
\begingroup%
  \makeatletter%
  \providecommand\color[2][]{%
    \errmessage{(Inkscape) Color is used for the text in Inkscape, but the package 'color.sty' is not loaded}%
    \renewcommand\color[2][]{}%
  }%
  \providecommand\transparent[1]{%
    \errmessage{(Inkscape) Transparency is used (non-zero) for the text in Inkscape, but the package 'transparent.sty' is not loaded}%
    \renewcommand\transparent[1]{}%
  }%
  \providecommand\rotatebox[2]{#2}%
  \newcommand*\fsize{\dimexpr\f@size pt\relax}%
  \newcommand*\lineheight[1]{\fontsize{\fsize}{#1\fsize}\selectfont}%
  \ifx\svgwidth\undefined%
    \setlength{\unitlength}{3401.57480315bp}%
    \ifx\svgscale\undefined%
      \relax%
    \else%
      \setlength{\unitlength}{\unitlength * \real{\svgscale}}%
    \fi%
  \else%
    \setlength{\unitlength}{\svgwidth}%
  \fi%
  \global\let\svgwidth\undefined%
  \global\let\svgscale\undefined%
  \makeatother%
  \begin{picture}(1,0.13333333)%
    \lineheight{1}%
    \setlength\tabcolsep{0pt}%
    \put(0,0){\includegraphics[width=\unitlength,page=1]{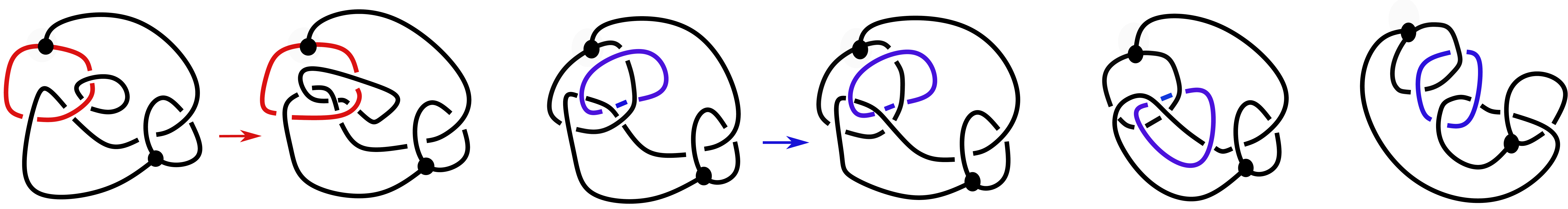}}%
    \put(0.3149972,0.0491486){\color[rgb]{0,0,0}\makebox(0,0)[lt]{\lineheight{1.25}\smash{\begin{tabular}[t]{l}$\simeq$\end{tabular}}}}%
    \put(0.66299296,0.04937967){\color[rgb]{0,0,0}\makebox(0,0)[lt]{\lineheight{1.25}\smash{\begin{tabular}[t]{l}$\simeq$\end{tabular}}}}%
    \put(0.83484731,0.04849773){\color[rgb]{0,0,0}\makebox(0,0)[lt]{\lineheight{1.25}\smash{\begin{tabular}[t]{l}$\simeq$\end{tabular}}}}%
  \end{picture}%
\endgroup%

\caption{$6_2$ and $6_4$ have homeomorphic complements.}
\label{fig:links_6_2_6_4}
\end{figure}
\begin{figure}[ht]
\def\svgwidth{0.85\columnwidth}
\begingroup%
  \makeatletter%
  \providecommand\color[2][]{%
    \errmessage{(Inkscape) Color is used for the text in Inkscape, but the package 'color.sty' is not loaded}%
    \renewcommand\color[2][]{}%
  }%
  \providecommand\transparent[1]{%
    \errmessage{(Inkscape) Transparency is used (non-zero) for the text in Inkscape, but the package 'transparent.sty' is not loaded}%
    \renewcommand\transparent[1]{}%
  }%
  \providecommand\rotatebox[2]{#2}%
  \newcommand*\fsize{\dimexpr\f@size pt\relax}%
  \newcommand*\lineheight[1]{\fontsize{\fsize}{#1\fsize}\selectfont}%
  \ifx\svgwidth\undefined%
    \setlength{\unitlength}{2806.2992126bp}%
    \ifx\svgscale\undefined%
      \relax%
    \else%
      \setlength{\unitlength}{\unitlength * \real{\svgscale}}%
    \fi%
  \else%
    \setlength{\unitlength}{\svgwidth}%
  \fi%
  \global\let\svgwidth\undefined%
  \global\let\svgscale\undefined%
  \makeatother%
  \begin{picture}(1,0.16161616)%
    \lineheight{1}%
    \setlength\tabcolsep{0pt}%
    \put(0,0){\includegraphics[width=\unitlength,page=1]{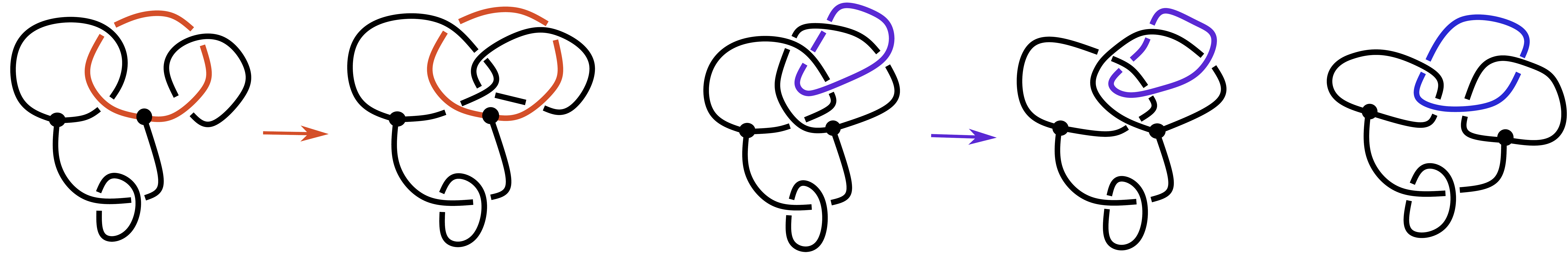}}%
    \put(0.39224496,0.06755226){\color[rgb]{0,0,0}\makebox(0,0)[lt]{\lineheight{1.25}\smash{\begin{tabular}[t]{l}$\simeq$\end{tabular}}}}%
    \put(0.80063883,0.0695534){\color[rgb]{0,0,0}\makebox(0,0)[lt]{\lineheight{1.25}\smash{\begin{tabular}[t]{l}$\simeq$\end{tabular}}}}%
  \end{picture}%
\endgroup%

\caption{$6_{11}$ and $6_{14}$ have homeomorphic complements.}
\label{fig:links_6_11_6_14}
\end{figure}

\begin{theorem}[\textbf{Unsplittability}]\label{teo:unsplittability}
Entries in Table \ref{tab:handlebodylinks} are all unsplittable.
\end{theorem}
\begin{proof}
In most cases $(5_1,6_1,6_2,6_3,6_4,6_7,6_9,6_{10})$ 
unsplittability follows by computing the linking number 
\cite{Miz:13} between pairs of components of a handlebody link.
There are a few cases where the linking number vanishes, and
we deal with these cases by 
computing the $ks_{\mathsf{A}_4}$- and 
$ks_{\mathsf{A}_5}$-invariants of the corresponding split
handlebody links (Table \ref{tab:uniqueness}).

If $6_5$ were split, then $6_5$ would be 
equivalent to the fake $6_5$ 
but this is not possible by Table \ref{tab:uniqueness}. 
In the case of $4_1$, $6_6$, $6_8$, if any of them were split, 
than it would be equivalent to ``split'' in Table \ref{tab:uniqueness}, but that is not the case.
A similar argument can be applied to $6_{11}$ and $6_{14}$: 
if one of them were split, it would be equivalent to the fake $6_{11}$, 
in contradiction to Table \ref{tab:uniqueness}. 
Lastly, we observe that $6_{12}$ and $6_{13}$ are non-split, for otherwise $4_1$ would be split.
\end{proof}

%
%

Below we recall the irreducibility test developed in \cite{BePaWa:20}.
A $r$-generator link is a link whose knot group, 
the fundamental group of its complement, is of rank $r$.  
\begin{lemma}\label{lemma:w_trivial_knot_A4_A5}
If the trivial knot is a factor of some 
factorization of a reducible $(n,1)$-handlebody link $\HL$,
then 
\begin{equation}\label{eq:trivial_knot_A4_A5}
12 \mid ks_{\mathsf{A}_4}(\HL)+6\cdot 3^n+2\cdot 4^n\quad \textbf{and}\quad
60 \mid ks_{\mathsf{A}_5}(\HL)+14\cdot 4^n+19\cdot 3^n +22\cdot 5^n.
\end{equation}  
\end{lemma} 
\begin{lemma}\label{lemma:w_2_generator_knot_A4}
If a $2$-generator knot is factor of some factorization of 
a reducible $(n,1)$-handlebody link $\HL$, then 
\begin{equation}\label{eq:2_generator_knot_A4}
12+24p \mid ks_{\mathsf{A}_4}(\HL)+(6+16p)\cdot 3^n+ (2+6p)\cdot 4^n,
\hspace*{.3em}
\textbf{where $p=0$ or $1$}.
\end{equation}
 
\end{lemma}
\begin{lemma}\label{lemma:w_2_generatore_link_A4}
If a $2$-component, $2$-generator link is a factor of
some factorization of 
a reducible $(n,1)$-handlebody link $\HL$, then 
\begin{equation}\label{eq:2_generator_link_A4}
48+24p \mid ks_{\mathsf{A}_4}(\HL)+(26+16p)\cdot 3^{n-1} + (8+6p)\cdot 4^{n-1},\hspace*{.3em}\textbf{where $p=0,1,2,3$ or $4$.}
\end{equation} 
 
\end{lemma}
From the above lemmas, one derives the following 
irreducibility test (see \cite{BePaWa:20} for more details),
making use of the Grushko theorem \cite{Gru:40}. 

\begin{corollary}[\textbf{Irreducibility test}]\label{reducibility_test}
A $3$-generator $(2,1)$-handlebody link is irreducible
if it fails to satisfy \eqref{eq:trivial_knot_A4_A5}; 
a $4$-generator $(2,1)$-handlebody link is irreducible
if it fails to satisfy \eqref{eq:2_generator_knot_A4}; 
a $4$-generator $(3,1)$-handlebody link or
a $5$-generator $(4,1)$-handlebody link
is irreducible if it fails to satisfy \eqref{eq:trivial_knot_A4_A5} 
and \eqref{eq:2_generator_link_A4}. 
\end{corollary}
 
\begin{theorem}[\textbf{Irreducibility}]\label{teo:irreducibility}
Entries in Table\ref{tab:handlebodylinks} are irreducible.
\end{theorem}
\begin{proof} 
Corollary \ref{reducibility_test}, together with Table \ref{tab:uniqueness}, 
shows that all but $6_9$, $6_{12}$ are irreducible. 
The irreducibility of $6_{12}$ and $6_9$ follows from
computing the linking number between each pair of components in 
each of them. Specifically, 
if $6_{12}$ (resp.\ $6_9$) is reducible, 
then either the trivial knot 
or a $2$-generator $2$-component link 
is a factor of some factorization of $6_{12}$ (resp.\ $6_9$).
For $6_{12}$, the former case 
is not possible by \eqref{eq:trivial_knot_A4_A5};
the latter impossible too, for otherwise the two solid 
torus components would have a trivial linking number.
The same argument implies that 
$6_9$ cannot have a $2$-generator $2$-component
link as a factor, and the trivial knot cannot be its factor either, since
the homomorphism of integral homology
\[H_1(V_1)\oplus H_1(V_2)\rightarrow H_1(\Compl{W})\] 
is onto, where $V_1, V_2$ are the solid torus components,
and $W$ the genus $2$ component.
\end{proof}

\begin{remark}\label{rem:irre_HL_re_complement}
The complement of $6_9$ is in fact $\partial$-reducible;
one can see this by performing the twist operation, 
indicated by the arrow in Fig.\ \ref{fig:6_9_fake6_9},
where it shows that its complement is homeomorphic 
to the complement of the order-$1$ connected sum (Definition \ref{def:one_sum})
between two Hopf links (Fig.\ \ref{fig:6_9_fake6_9}, right).
In the connected case, 
no irreducible handlebody knot of genus $2$ admits a
$\partial$-reducible complement \cite[Theorem $1$]{Tsu:75},  
but when the genus is larger than $2$, such handlebody knots exist 
\cite[Example $5.5$]{Suz:75}, \cite[Section $5$]{Tsu:75}.
For a $(n,1)$-handlebody link, we suspect that $6_9$ attains 
the lowest possible $n$ for such a phenomenon to happen. 
\begin{figure}[h]
\def\svgwidth{0.95\columnwidth}
\begingroup%
  \makeatletter%
  \providecommand\color[2][]{%
    \errmessage{(Inkscape) Color is used for the text in Inkscape, but the package 'color.sty' is not loaded}%
    \renewcommand\color[2][]{}%
  }%
  \providecommand\transparent[1]{%
    \errmessage{(Inkscape) Transparency is used (non-zero) for the text in Inkscape, but the package 'transparent.sty' is not loaded}%
    \renewcommand\transparent[1]{}%
  }%
  \providecommand\rotatebox[2]{#2}%
  \newcommand*\fsize{\dimexpr\f@size pt\relax}%
  \newcommand*\lineheight[1]{\fontsize{\fsize}{#1\fsize}\selectfont}%
  \ifx\svgwidth\undefined%
    \setlength{\unitlength}{3685.03937008bp}%
    \ifx\svgscale\undefined%
      \relax%
    \else%
      \setlength{\unitlength}{\unitlength * \real{\svgscale}}%
    \fi%
  \else%
    \setlength{\unitlength}{\svgwidth}%
  \fi%
  \global\let\svgwidth\undefined%
  \global\let\svgscale\undefined%
  \makeatother%
  \begin{picture}(1,0.11538462)%
    \lineheight{1}%
    \setlength\tabcolsep{0pt}%
    \put(0,0){\includegraphics[width=\unitlength,page=1]{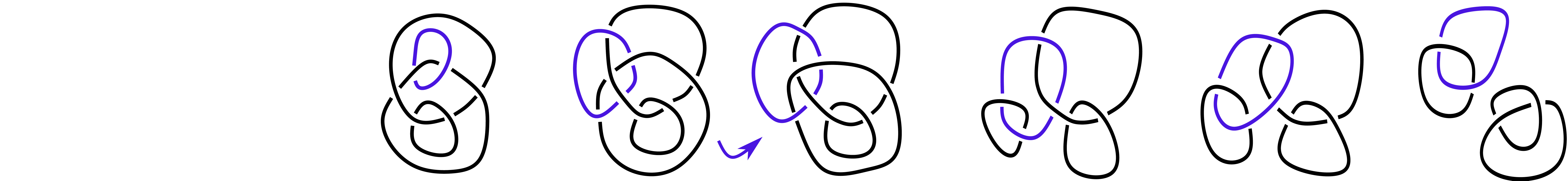}}%
    \put(0.32870274,0.05075979){\color[rgb]{0,0,0}\makebox(0,0)[lt]{\lineheight{1.25}\smash{\begin{tabular}[t]{l}$\simeq$\end{tabular}}}}%
    \put(0.58987018,0.05068951){\color[rgb]{0,0,0}\makebox(0,0)[lt]{\lineheight{1.25}\smash{\begin{tabular}[t]{l}$\simeq$\end{tabular}}}}%
    \put(0.7321759,0.0485348){\color[rgb]{0,0,0}\makebox(0,0)[lt]{\lineheight{1.25}\smash{\begin{tabular}[t]{l}$\simeq$\end{tabular}}}}%
    \put(0.87522508,0.04580496){\color[rgb]{0,0,0}\makebox(0,0)[lt]{\lineheight{1.25}\smash{\begin{tabular}[t]{l}$\simeq$\end{tabular}}}}%
    \put(0,0){\includegraphics[width=\unitlength,page=2]{links6_9_fake6_9.pdf}}%
    \put(0.20714179,0.04732422){\color[rgb]{0,0,0}\makebox(0,0)[lt]{\lineheight{1.25}\smash{\begin{tabular}[t]{l}$\simeq$\end{tabular}}}}%
    \put(0.0912713,0.05152589){\color[rgb]{0,0,0}\makebox(0,0)[lt]{\lineheight{1.25}\smash{\begin{tabular}[t]{l}$\simeq$\end{tabular}}}}%
    \put(0,0){\includegraphics[width=\unitlength,page=3]{links6_9_fake6_9.pdf}}%
  \end{picture}%
\endgroup%
 
\caption{$6_9$ and fake $6_9$.}
\label{fig:6_9_fake6_9}
\end{figure} 
\end{remark}

\nada{
\begin{remark}\label{rem:reducibility}
This result also implies that Table \ref{tab:reducible} is complete.
Indeed by contradiction there would exist a diagram with edge connectivity two or three corresponding to a reducible
link with a larger number of crossings.
Such diagram would result from the software code (if the link is of type-3) or would be a knot sum of spatial graphs
if it is of type-2.
This would become apparent when checking such diagrams against their reducibility, but this is not the case. 
\end{remark}
} 


\section{Completeness}\label{sec:completeness}
This section discusses completeness
of Table \ref{tab:handlebodylinks}. 
Recall first that a minimal diagram of a 
non-split, irreducible handlebody link 
has either $2$- or $3$-connectivity.
IH-minimal diagrams with $3$-connectivity 
are obtained from a software code,
and IH-minimal diagrams with $2$-connectivity 
are recovered by knot sum of spatial graphs. 
 
\subsection{Minimal diagrams with 3-connectivity}
We consider plane graphs with two trivalent
vertices and up to six quadrivalent vertices
satisfying the properties:   
\begin{enumerate}
\item
each of them has edge-connectivity $3$ as an abstract graph,  
\item
their double arcs can only connect two quadrivalent vertices as abstract graphs, and  
\item
their double arcs only form a ``bigon'' (a polygon with two sides;
the case `i' in Fig.\ \ref{fig:loops_double_arcs}) as plane graphs.
\end{enumerate}
The reason of considering only double arcs connecting two quadrivalent vertices with a bigon configuration
is because all the other cases lead to 
either non-R-minimal diagrams or diagrams
with connectivity less than $3$ (see Fig.\ \ref{fig:loops_double_arcs},
where ``d, e, f, g, h'' illustrate those double arcs 
connecting at least one trivalent vertex and ``j, k, l''  
those connecting two quadrivalent 
vertices with a non-bigon configuration.)

\begin{figure}[h]
\includegraphics[height=0.1\textwidth]{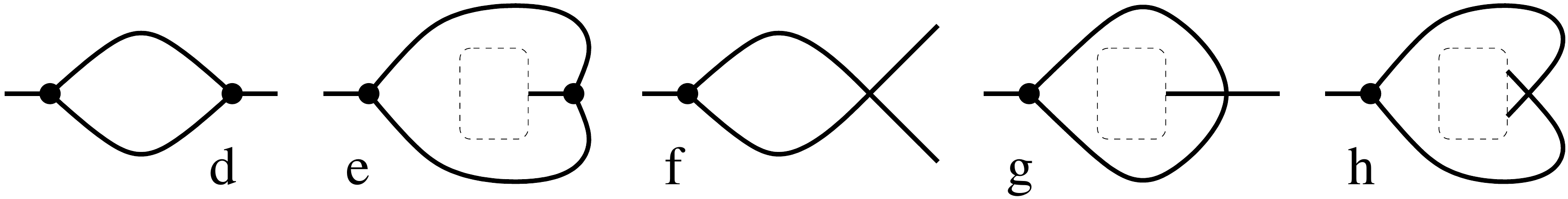}
\includegraphics[height=0.1\textwidth]{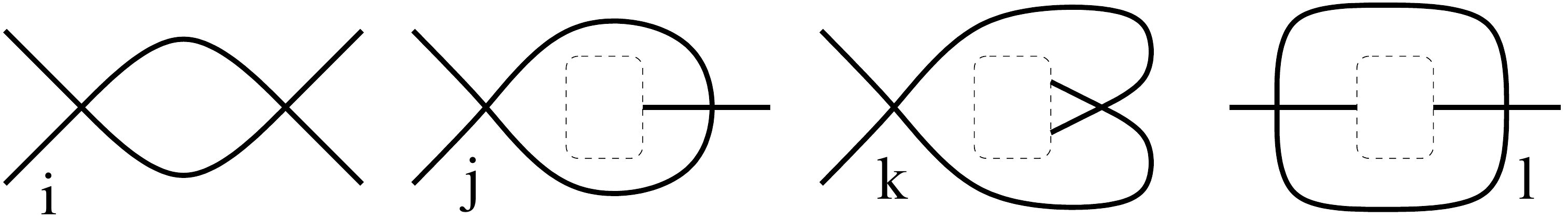}
\caption{Possible configurations for loops and double arcs.}
\label{fig:loops_double_arcs}
\end{figure}
We enumerate such plane graphs by the software code, and 
then recover diagrams from these plane graphs by adding an
over- or under-crossing to each quadrivalent vertex.
Note that the number ($n$ in Table \ref{tab:numfromcode}) 
of components of the associated spatial graphs
is independent of how over/under-crossings are chosen.
To provide a glimpse of how the code works, we record in
Table \ref{tab:numfromcode} the number of 
such plane graphs with $c$ quadrivalent
vertices for each $n$. To recover 
$(n,1)$-handlebody links with $n>1$ represented by 
IH-minimal diagrams with $3$-connectivity,
we need to consider 
the cases with $n>1$ in Table \ref{tab:numfromcode}.
On the other hand, to produce $(n,1)$-handlebody links
represented by IH-minimal diagrams with $2$-connectivity, 
spatial graphs admitting an R-minimal diagram with 
$3$-connectivity up to $4$ crossings are required; thus 
all cases with $c\leq 4$ have to be examined.

\begin{table}[h!]
  \begin{center}
    \caption{Plane graphs given by the code.
     }
    \label{tab:numfromcode}
    \begin{tabular}{c|rrr|r}
      $c$ & $n=1$ & $n=2$ & $n=3$ & \textbf{total} \\
      \hline
      2 &   1 &    &   &   1 \\
      3 &   2 &  1 &   &   3 \\
      4 &   8 &  2 &   &  10 \\
      5 &  29 &  8 &   &  37 \\
      6 & 144 & 34 & 3 & 181 \\
    \end{tabular}
  \end{center}
\end{table}

\medskip
\noindent  
\textbf{IH-minimal diagrams.} 
We examine IH-minimality of diagrams produced by
plane graphs with $n\geq 2$,
and discard those obviously not IH-minimal.
This excludes all diagrams
produced by the code up to $5$ crossings  
(Table \ref{tab:detailsfor5}), but for diagrams with $6$
crossings, some diagrams are potentially IH-minimal: they 
represent handlebody links $6_1$, $6_2$, $6_3$ or $6_9$ in Table \ref{tab:handlebodylinks}.
\begin{lemma}\label{lm:three_connected_IH_minimal}
An IH-minimal diagram with $3$-connectivity has crossing number 
$c\geq 6$, and if $c=6$, it represents a handlebody link equivalent to 
$6_1$, $6_2$, $6_3$ or $6_9$, up to mirror image.
\end{lemma} 
Note that we cannot
conclude diagrams of $6_1, 6_2, 6_3$ and $6_9$ in Table \ref{tab:handlebodylinks}
are IH-minimal yet,  as they 
might admit diagrams with $2$-connectivity
and fewer crossings.

\medskip
\noindent
\textbf{R-minimal diagrams.}
To produce minimal diagrams with $2$-connectivity up to $6$ crossings,
we need R-minimal diagrams up to $4$ crossings.
Inspecting R-minimality of diagrams produced
by the code (Table \ref{tab:detailsfor1_4}) gives us the following lemma.
\begin{lemma}\label{lm:three_connected_R_minimal}
An R-minimal diagram with $3$-connectivity and crossing number less than $5$   
represents one of the spatial graphs in Table \ref{tab:graphs}, up to mirror image.       
\begin{table}[ht]
\caption{Spatial graphs up to four crossings. 
}
\label{tab:graphs}
\includegraphics[height=0.12\textwidth]{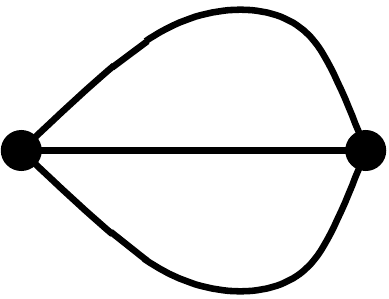}$\op{G0_1}$
\includegraphics[height=0.15\textwidth]{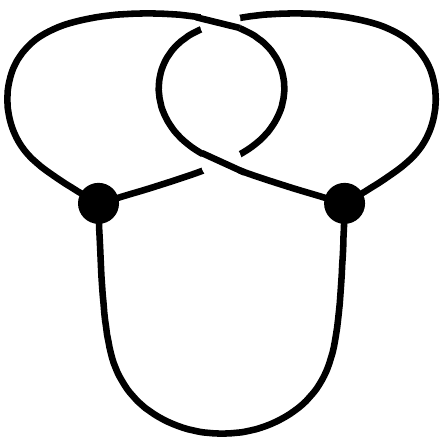}$\op{G2_1}$
\includegraphics[height=0.15\textwidth]{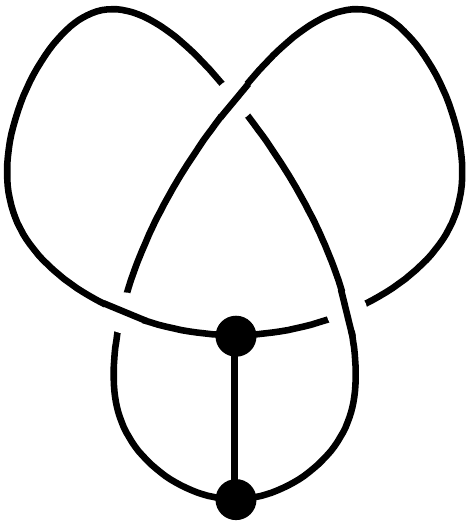}$\op{G3_1}$
\includegraphics[height=0.15\textwidth]{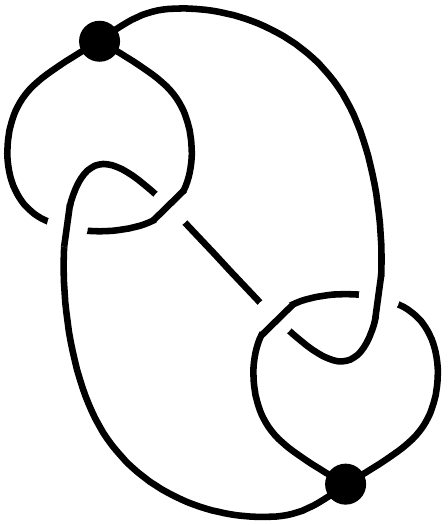}$\op{G4_1}$\\
\vspace*{1.5em} 
\includegraphics[height=0.15\textwidth]{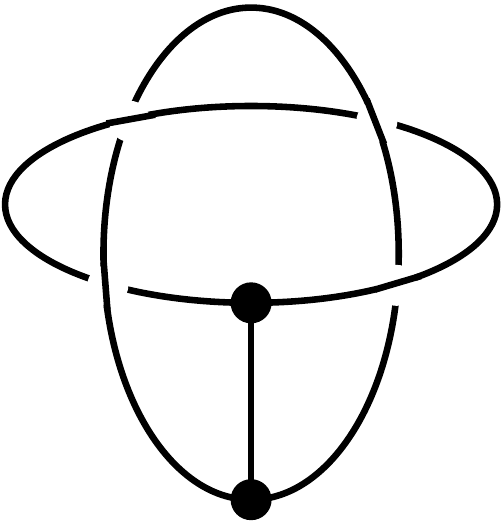}$\op{G4_2}$
\includegraphics[height=0.15\textwidth]{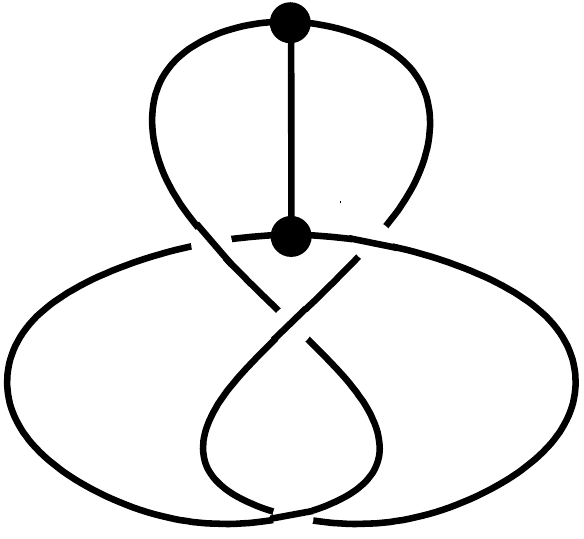}$\op{G4_3}$
\includegraphics[height=0.15\textwidth]{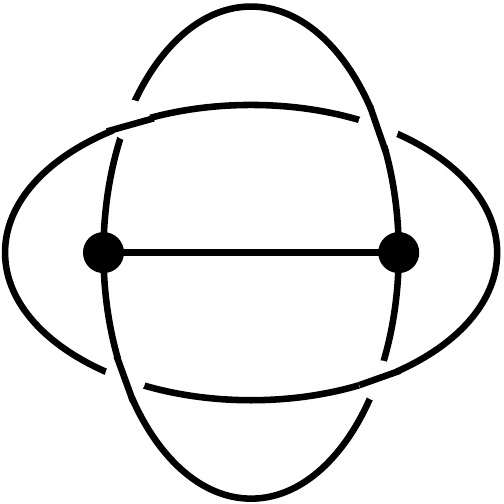}$\op{G4_4}$
\includegraphics[height=0.15\textwidth]{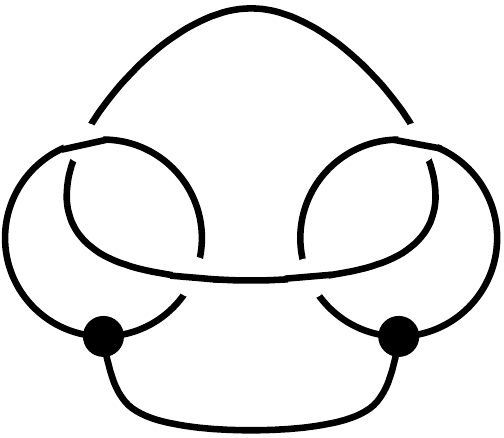}$\op{G4_5}$
\end{table}  
\end{lemma}

\subsection{Minimal diagrams with 2-connectivity}\label{subsec:typetwo}
Recall a diagram $D$ with $2$-conn\-ectivity 
can be decomposed into 
finitely many simpler tangle diagrams such that each 
associated diagram of spatial graphs has $3$- or $4$-connectivity
(Fig.\ \ref{fig:decomposition_type_2}).
Furthermore, if $D$ is R-minimal, each induced spatial graph 
diagram is also R-minimal. In particular, an IH-minimal diagram 
with $2$-connectivity
can be recovered by performing the order-2 vertex connected sum
between spatial graphs admitting a minimal diagram 
with $k$-connectivity, $k>2$.
Since we are interested in $(n,1)$-handlebody links, 
only one summand is a spatial graph 
with two trivalent vertices,
and the rest are links
admitting a minimal diagram with $4$-connectivity.
\nada{
If we think of this operation as performed on spatial graphs that project onto plane graphs we
are adding (in the sense of knot sum) knots or links to a possibly knotted handcuff or theta
curve possibly tangled with one or more embedded circles.

In order to list all possible results of knot sums we need to start from a $3$-edge-connected
graph corresponding to the diagram of some spatial graph where a (knotted) handcuff or a theta curve is
tangled to zero of more embedded circles.

Reidemeister moves can be safely performed separately on the summands of a knot sum (after adding
over/under information on crossings):  they will be still applicable on the sum result.
On the contrary, the IH-move does not in general survive the knot sum operation.
}
Note that the simplest minimal diagram with $4$-connectivity 
represents the Hopf link, and since we only consider minimal diagrams 
up to $6$ crossings, there are at most three link summands. 
Thus, IH-minimal diagrams with $2$-connectivity 
can be recovered by considering the seven possible configurations below:
\begin{multicols}{2}
\begin{enumerate}
\item
$G \# L_1$,
\item
($G \# L_1) \# L_2$,
\item
$G \# (L_1 \# L_2)$,\\
\item
(($G \# L_1) \# L_2) \# L_3$,
\item
($G \# L_1) \# (L_2 \# L_3)$,
\item
($G \# (L_1 \# L_2)) \# L_3$,
\item
$G \# ((L_1 \# L_2) \# L_3)$,
\end{enumerate}
\end{multicols}
\noindent
where $G$ is a spatial graph admitting a minimal diagram
with $3$-connectivity, 
and $L_i$ is a link admitting a minimal diagram with $4$-connectivity. 
In general it is not known if 
a minimal diagram with $4$-connectivity always represents 
a prime link; it is the case, however, when the crossing number
is less than $5$. 
In fact, there are only four 
minimal diagrams with $4$-connectivity up to $4$ crossings, and 
they represent the Hopf link,
the trefoil knot, the figure eight, and Solomon's knot (L4a1), respectively.

\textbf{Cases $4$ through $7$} are easily dealt with 
since $G$ must have no crossings, and hence it is
the trivial theta curve $\op{G0_1}$, and thus each $L_i$ is necessarily 
the Hopf link, so the knot sums actually
consist in `inserting a ring' somewhere 
to the result of the previous knot sums.
To produce
irreducible handlebody links
there is only one possibility, that is, 
adding one Hopf link to each of the three arcs of the trivial theta curve, and
this gives us entry $6_{15}$ in Table \ref{tab:handlebodylinks}.

\textbf{Cases $2$ and $3$} forces $G$ to have $2$ crossings at most.  
It cannot have zero crossings (trivial
theta curve), for otherwise, we could only get diagrams
representing reducible handlebody links.
Note also that there is no $R$-minimal diagram with $1$ crossing.
Now, there is only one $R$-minimal diagram with two crossings, 
this is, entry G$2_1$ in Table \ref{tab:graphs} 
(Moriuchi's $2_1$ in \cite{Mor:07b}).

Now, to add two Hopf links to it, 
i.e. to place two rings successively, we observe that one of them must
be placed around the connecting arc of the handcuff graph 
by irreducibility.
The second ring can be placed in three inequivalent ways, 
which yield entries $6_{12}$, $6_{13}$ and
$6_{14}$ of Table \ref{tab:handlebodylinks}.

\textbf{Case $1$} is more complicated, and
we divided it into subcases based on the crossing number   
$c:=c(G)$. The case $c = 0$ is immediately
excluded by irreducibility, so three possibilities remain: $c \in \{2,3,4\}$.

%
\textit{Subcase $c(G) = 2$}.
$G$ is necessarily G$2_1$ in Table \ref{tab:graphs}, and    
$L$ cannot be a knot. 
Since the crossing number of $L$ cannot exceed $4$, $L$ is
either L2a1 (Hopf link) or L4a1 (Solomon's knot).
In either case, $L$ is to be added to 
the connecting arc of the handcuff graph 
to produce irreducible handlebody links, yielding
entries $4_1$ and $6_8$ in Table \ref{tab:handlebodylinks}.

\textit{Subcase $c(G) = 3$}.
$G$ is necessarily G$3_1$   
(Moriuchi's theta curve $3_1$ in \cite{Mor:09a}),
so $L$ cannot be a knot, and hence is the Hopf link.
There is only one place to add $L$ by irreducibility, and 
this leads to entry $5_1$ in Table \ref{tab:handlebodylinks}.

\textit{Subcase $c(G) = 4$}.
In this case, $L$ can only be the Hopf link; 
and there are five possible spatial graphs for $G$, namely
G$4_1$, G$4_2$, G$4_3$, G$4_4$, and G$4_5$: 
\begin{itemize}
\item
For G$4_1$ in Table \ref{tab:graphs}
(Moriuchi's non-prime handcuff graph $2_1 \#_3 2_1$ \cite{Mor:09b}), 
there are two inequivalent ways to add $L$ which produce 
entries $6_4$ and $6_5$.
\item
For G$4_2$ and G$4_3$ in Table \ref{tab:graphs} 
(Moriuchi's prime handcuff graph $4_1$ \cite{Mor:07b}
and prime theta-curve $4_1$ \cite{Mor:09a}, respectively), 
there is only one way to 
add the Hopf link in each case by irreducibility, and 
this gives 
$6_6$, $6_7$ in Table \ref{tab:handlebodylinks}, respectively.
\item 
For G$4_4$ and G$4_5$ in Table \ref{tab:graphs}, again by 
irreducibility, there is only one way to add the Hopf link in each case, 
which gives us $6_{10}$ and $6_{11}$ in Table \ref{tab:handlebodylinks},
respectively.
\end{itemize}

 
We summarize the discussion above in the following: 
\begin{lemma}\label{lm:two_connected_IH_minimal}
A non-split, irreducible handlebody link admitting an 
IH-minimal diagram with $2$-connectivity and crossing number $\leq 6$ 
is equivalent, up to mirror image, to one of the following handlebody links: 
\begin{equation}\label{list:handlebody_link_type_2}
 4_1, 5_2, 6_4, 6_5, 6_6, 6_7, 6_8, 6_{10}, 6_{11}, 6_{12}, 6_{13}, 6_{14}.
\end{equation}  
\end{lemma}

By Lemma \ref{lm:three_connected_IH_minimal}, if any of  
\eqref{list:handlebody_link_type_2}  
admits an IH-minimal diagram with $3$-connectivity, 
it is equivalent to one of
$6_1,6_2,6_3$, $6_9$, 
while by Lemma \ref{list:handlebody_link_type_2} 
if $6_1,6_2,6_3$ or $6_9$ admits an
IH-minimal diagram with $2$-connectivity and 
less than $6$ crossings, it is equivalent to   
$4_1$ or $5_1$, but neither situation can happen by Theorem \ref{teo:uniqueness}.
\begin{corollary}
Diagrams in Table \ref{tab:handlebodylinks} are all IH-minimal.
\end{corollary}
\nada{
\begin{corollary}
List \ref{list:handlebody_link_type_2} 
enumerates all handlebody links of type $2$ up to $6$ crossings
and their diagrams in Fig.\ \ref{tab:handlebodylinks}
are minimal.
\end{corollary} 
}





 

\section{Chirality}\label{sec:chirality}
\subsection{Decomposable links}
We consider  
order-$2$ connected sum (compare with Definition \ref{def:one_sum}) 
for handlebody-link-disk pairs. 
A handlebod-link-disk pair is a handlebody link $\HL$
with an oriented incompressible disk $D\subset \HL$. 
A trivial knot with a meridian disk 
is considered as the trivial handlebody-link-disk pair.
\begin{definition}[\textbf{Order-2 connected sum}]\label{def:two_sum}
Given two handlebody-link-disk pairs 
$(\HL_1,D_1)$, $(\HL_2,D_2)$ 
the order-2 connected sum $(\HL_1,D_1)\#(\HL_2,D_2)$
is obtained as follows: first choose 
a $3$-ball $B_i$ of $D_i$ in $\sphere$ for each $i$ 
with $\mathring{B_i}\cap \HL_i$ a tubular neighborhood 
$N(D_i)$ of $D_i$ in
$\HL_i$; next, identify $\overline{N(D_i)}$ with $D_i\times [0,1]$ via
the orientation of $D_i$. Then $(\HL_1,D_1)\#(\HL_2,D_2)$
is given by removing $\mathring{B_i}$ and
gluing the resulting manifolds via an orientation-reversing
homeomorphism:
\[
h: \partial (\Compl{B_1}) \rightarrow \partial (\Compl{B_2})\quad
\textbf{with $h(D_1\times\{j\})=D_2\times\{k\}$, $k\equiv j+1$ mod $2$}.
\]
A non-split, irreducible handlebody link
is \textit{decomposable} if it is 
equivalent to an order-2 connected sum of non-trivial handlebody-link-disk pairs.
\end{definition}
\begin{figure}[ht] 
\def\svgwidth{0.9\columnwidth}
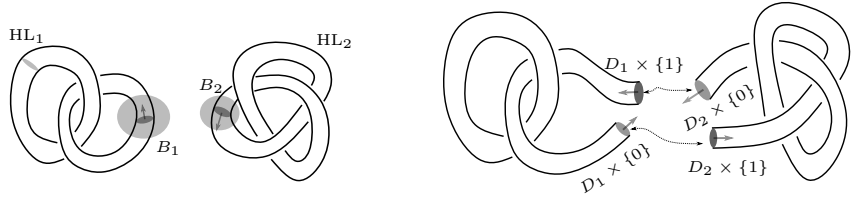   
\caption{Knot sum of handlebody-link-disk pairs}
\label{fig:knotsum}
\end{figure}
Decomposability is reflected in 
minimal diagrams in most examples here, thus we state the following conjecture.  
\begin{conjecture}\label{conj:decomp_is_of_type_2}
Minimal diagrams of a decomposable handlebody link 
have $2$-connectivity.
\end{conjecture}

\begin{lemma}\label{lm:critera_decomposable} 
For a non-split, irreducible handlebody link $\HL$,  
the following statements are equivalent:\\ 
\textbullet\ $\HL$ is decomposable;\\
\textbullet\ 
there exists a $2$-sphere $\Stwo$ in $\sphere$
such that $\Stwo$ and $\HL$ intersect at two incompressible 
disks in $\HL$
and neither of $\overline{B_i\setminus \HL}$, $i=1,2$,
is a solid torus,
where $B_1, B_2$ are components of $\Compl{\Stwo}$;\\
\textbullet\      
$\Compl{\HL}$ admits an incompressible, 
non-boundary parallel (or $\partial$-incompressible)
annulus $A$ with $\partial A$ inessential in $\HL$. 
\end{lemma}
\begin{proof}
This follows from the definition of ($\partial$-) incompressibility.
\end{proof}
The annulus $A$ in Lemma \ref{lm:critera_decomposable} 
is called a decomposing annulus of $\HL$. Similarly,
a decomposiing annulus $A$ of a handlebody-link-disk
pair $(\HL,D)$ is 
an incompressible, non-boundary parallel 
annulus $A$ with $\partial A$ inessential in $\HL$ and
disjoint from $\partial D$; note that $\HL$
in $(\HL,D)$ could be reducible. 

We now prove a unique decomposition theorem for handlebody links
with no genus $g>2$ component; a unique 
decomposition theorem for handlebody knots of arbitrary genus 
is given \cite[Appendix $B$]{KodOzaGor:15} (see also \cite{IshKisOza:15}).  
\begin{theorem}\label{teo:uniqueness_decomposition_special_case}
Given a non-split, irreducible handlebody link $\HL$,
suppose no component of $\op{HL}$ has genus greater than 2, and 
$A,A'$ are decomposing annuli inducing  
\begin{equation}\label{eq:two_decomp_of_hl}
\HL\simeq (\HL_1,D_1)\#(\HL_2,D_2), \;
\HL\simeq (\HL_1',D_1')\# (\HL_2',D_2'),  \textbf{ respectively}.
\end{equation}
If $(\Compl{\HL_i},D_i),i=1,2$, 
admits no decomposing annulus.
Then $A,A'$ are isotopic, in the sense that 
there exists an ambient isotopy $f_t:\sphere\rightarrow \sphere$ 
fixing $\op{HL}$ with $f_1(A)=A'$.  
\end{theorem}
\begin{proof}
Note first that if $A$, $A'$ are disjoint,
then the assumption implies that they must be parallel and hence isotopic. 
Suppose $A\cap A' \neq \emptyset$.
Then we isotopy $A$ such that 
the number of components of $A\cap A'$ is minimized. 

\noindent
\textbf{Claim: any circle or arc in $A\cap A'$ is essential in both $A$ and $A'$.}
Observe first that a 
circle component $C$ of $A\cap A'$
is either essential 
or inessential in both $A$ and $A'$, 
for assuming otherwise
would contradict the incompressibility of $A,A'$. 
%
%
%
Suppose $C$ is inessential in both $A$ and $A'$, and  
is innermost in $A'$. Then $C$ bounds disks $D,D'$ 
in $A,A'$, respectively. 
Since $\HL$ is non-split, 
$D\cup D'$ bounds a $3$-ball $B$ in
$\Compl{\HL}$. Isotopy $D$ across $B$ to $D'$
induces a new annulus $A$ isotopic to the original one with
$A \cap A'$ having less components, contradicting the minimality.

Similarly, if $l$ is an arc component of $A \cap A'$, 
then it is either essential or inessential in both $A$ and $A'$,
for assuming otherwise would contradict the $\partial$-incompressibility
of $A, A'$.
Suppose $l$ is inessential in both $A$ and $A'$ and 
an innermost arc in $A'$.
Then $l$ cuts off a disk $D'$ from $A'$
and a disk $D$ from $A$. 
Let $\hat{D}:=D\cup D'$.
If $\partial \hat{D}$
is inessential, then we can remove the intersection $l$
by isotopying $A$ across the ball bounded by $\hat{D}$
and the disk bounded by 
$\partial \hat{D}$ in $\partial \op{HL}$, contradicting the minimality.
If $\partial \hat{D}$ is essential,
then isotopying $\hat{D}$, we can disjoin $\hat{D}$ from $A$. 
 
Now, it may be assumed that $D_1$ 
is in a genus one component of $\op{HL}_1$, 
and hence $D_2$ is in a component of $\op{HL}_2$ 
with genus $\leq 2$. 
Since $\partial \hat{D}$ is essential, 
$\hat{D}$ has to be in a genus $2$
component of $\op{HL}_2$ containing $D_2$. 
Because $\partial\hat{D}\cap D_2=\emptyset$, 
if $\partial \hat{D}$ is essential on the boundary of
the embedded solid torus $\big(\HL_2\setminus N(D_2)\big)\subset\sphere$, $\partial \hat{D}$ would be its longitude, where $N(D_2)$ is a tubular neighborhood of $D_2$, 
disjoint from $\hat{D}$, in $\op{HL}_2$.
Particularly, $\HL_2$ and therefore $\HL$ would be reducible,
a contradiction. 
On the other hand, if
$\partial \hat{D}$ bounds 
a disk on $\partial \big(\HL_2\setminus N(D_2)\big)$ 
that contains some components 
of $\partial\overline{N(D_2)}$,  
then $\partial \hat{D}$
is inessential in $\HL_2$, and hence in $\HL$, again
contradicting the irreducibility of
$\HL$.

The claim is proved, which also implies 
$A \cap A'$ contains either circles or arcs. 

\noindent
\textbf{No essential circles.}
Suppose $C$ is an essential circle,
and a closest circle to $\partial A'$.
Let $R'$ be the annulus cut off by $C$ from $A'$ 
with $A\cap R'=C$ and 
$R$ an annulus cut off by $C$ from $A$.
We isotopy the incompressible annulus 
$\hat{R}:=R\cup R'$ away from $A$. Since
components of $\partial \hat{R}$ are inessential in $\op{HL}$, 
by the assumption, $\hat{R}$ is either parallel to $A$ or 
boundary-parallel. In the former case, 
replacing $A$ with $R$ leads to a contradiction 
since $R\cap A'$ has less components
than $A\cap A'$.
In the latter case,   
isotopying $R$ through the solid torus
$V$ bounded by $\hat{R}$ and the part of $\partial \HL$
parallel to $\hat{R}$ gives a new $A$ 
isotopic to the original one but with less components in $A\cap A'$, 
contradicting the minimality.

\noindent
\textbf{No essential arcs.}     
Suppose $l_1$ is an essential arc. Then choose
the essential arc $l_2$ next to $l_1$ in $A'$ 
such that the disk $D'$ cut off by 
$l_1,l_2$ from $A'$ has $D'\cap A=l_1\cup l_2$ and 
is on the side of $A$ containing components of $\op{HL}_1$.
Let $D$ be a disk cut off by $l_1,l_2$ from $A$.
It may be assumed, by pushing $D$ away from $A$, 
that $D\cup D'$ is disjoint from $A$, and hence
is on the genus $1$ complement of $\op{HL}_1$ containing $D_1$.
Since $\hat{A}:=D\cup D'$ is disjoint from $D_1$, it 
is necessarily an annulus,
for if it were a M\"obious band,  
we would get a non-orientable surface embedded in $\sphere$.
%
Furthermore, each component of $\partial \hat{A}$ is 
necessarily inessential in $\op{HL}_1$, so
it either bounds a meridian disk
or is inessential in $\partial \op{HL}_1$.
Note also it cannot be the case that one component of $\partial \hat{A}$
is essential in $\partial\op{HL}_1$
and the other inessential by the irreducibility of $\op{HL}$.
Suppose both components are inessential in $\partial \op{HL}_1$.
Then $\hat{A}$, together with disks on $\partial \op{HL}_1$
bounded by $\partial \hat{A}$, bounds a $3$-ball, with which
we can isotopy $A$ to remove the intersection $l_1,l_2$,
contradicting the minimality.
Suppose both components bound meridian disks in $\op{HL}_1$.
Then $\tilde{A}=D^c\cup D'$ has 
$\partial \tilde{A}$ inessential in $\partial \op{HL}_1$, where
$D^c=\overline{A\setminus D}$.
Thus we reduce it to the previous case.
\nada{
Then $\hat{A}$ is incompressible and hence it either
parallel to $A$ or boundary-parallel by the assumption,
contradicting the minimality. 
}
\nada{
Then we replace $D$ with the closure of another component
of $\overline{A\setminus l_1\cup l_2}$ and
reduce it to the previous case (i.e. 
components of $\partial \hat{A}$ are  
inessential in $\partial \op{HL}_1$).
}
\nada{
component $M$ of $\Compl{\HL}\setminus A_1$ with 
$M\cap \partial \HL$ are some tubes.
Let $D',D''$ be the disks of $A_1\setminus l\cup l'$,
and $A', A''$ be the annuli $D\cup D'$, $D\cup D''$.
Then either components of $\partial A'$ or components of $\partial A''$
are essential in $\partial \HL$, say $\partial A'$. 
Let $P,P'\subset \partial \HL$ be disks bounded by $\partial A'$.
Then $A'\cup P cup P'$ bounds a $3$-ball $B $in $\Compl{\HL}$;
isotopy $D'$ across $B$ to $D$ gives us a new annulus $A_1'$
with $A_1'\cap A_2$ has less components than $A_1\cap A_2$. 
}
\end{proof}

\draftMMM{That would be nice, however to make the completeness proof work we need to define type-2 as
those handlebody links that admit a *minimal* diagram with edge connectivity two, which unfortunately
is a bit different: we could have a type-2 with the geometric definition having no minimal diagram
with edge connectivity 2!}
\draftYYY{That is right; we have no proof for that.
On the other hand, if I am not mistaken,
those handlebody links (not reducible, not splittable)
that do not have such a minimal diagram should be recovered by the code, right?
That is $6_1$, $6_2$, $6_3$, and $6_4$ are potentially minimal
diagrams of some 2-sums.
}
\draftMMM{Yes, indeed.  Is there a way to prove that this is not the case?}
\draftYYY{I am also wondering; 
it would be great if we could add this to the main text. 
One possible way is to prove it via invariant that are additive with respect 
to 2-sum, maybe we could try using Alexander invariant!?}

\subsection{Chirality}
We divide the proof of 
Theorem \ref{teo:chiral_hl} into two lemmas.

\begin{lemma}\label{lm:achiral}
All handlebody links except for  
$5_1$, $6_3$, $6_6$, $6_7$, $6_8$, $6_{10}$   
in Table \ref{tab:handlebodylinks} are achiral.
\end{lemma}
\begin{proof}
Except $6_2$, equivalences between these handlebody links and 
their mirror images are easy to construct; 
an equivalence between $6_2$ and r$6_2$
is depicted in Fig.\ \ref{fig:chirality_6_2} 
\begin{figure}[ht]
\def\svgwidth{0.9\columnwidth}
\begingroup%
  \makeatletter%
  \providecommand\color[2][]{%
    \errmessage{(Inkscape) Color is used for the text in Inkscape, but the package 'color.sty' is not loaded}%
    \renewcommand\color[2][]{}%
  }%
  \providecommand\transparent[1]{%
    \errmessage{(Inkscape) Transparency is used (non-zero) for the text in Inkscape, but the package 'transparent.sty' is not loaded}%
    \renewcommand\transparent[1]{}%
  }%
  \providecommand\rotatebox[2]{#2}%
  \newcommand*\fsize{\dimexpr\f@size pt\relax}%
  \newcommand*\lineheight[1]{\fontsize{\fsize}{#1\fsize}\selectfont}%
  \ifx\svgwidth\undefined%
    \setlength{\unitlength}{3118.11023622bp}%
    \ifx\svgscale\undefined%
      \relax%
    \else%
      \setlength{\unitlength}{\unitlength * \real{\svgscale}}%
    \fi%
  \else%
    \setlength{\unitlength}{\svgwidth}%
  \fi%
  \global\let\svgwidth\undefined%
  \global\let\svgscale\undefined%
  \makeatother%
  \begin{picture}(1,0.14545455)%
    \lineheight{1}%
    \setlength\tabcolsep{0pt}%
    \put(0,0){\includegraphics[width=\unitlength,page=1]{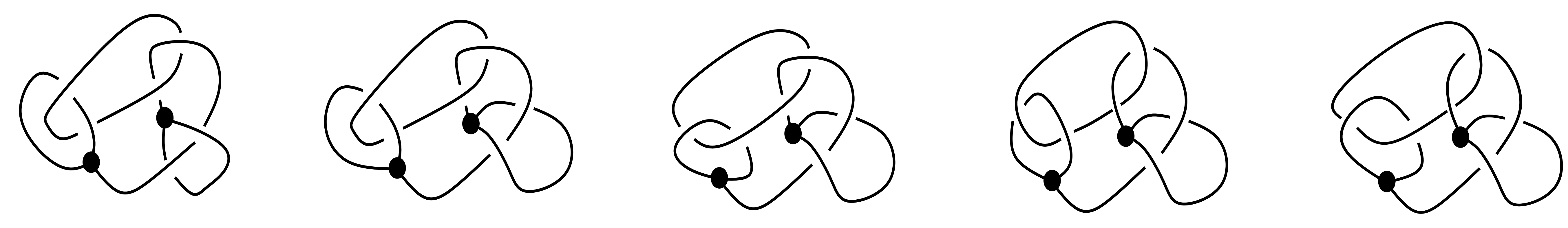}}%
    \put(0.15675102,0.06493506){\color[rgb]{0,0,0}\makebox(0,0)[lt]{\lineheight{1.25}\smash{\begin{tabular}[t]{l}$\simeq$\end{tabular}}}}%
    \put(0.37454724,0.06578114){\color[rgb]{0,0,0}\makebox(0,0)[lt]{\lineheight{1.25}\smash{\begin{tabular}[t]{l}$\simeq$\end{tabular}}}}%
    \put(0.5790146,0.067691){\color[rgb]{0,0,0}\makebox(0,0)[lt]{\lineheight{1.25}\smash{\begin{tabular}[t]{l}$\simeq$\end{tabular}}}}%
    \put(0.79010313,0.06769099){\color[rgb]{0,0,0}\makebox(0,0)[lt]{\lineheight{1.25}\smash{\begin{tabular}[t]{l}$\simeq$\end{tabular}}}}%
  \end{picture}%
\endgroup%
 
\caption{$6_2$ and its mirror image.}
\label{fig:chirality_6_2}
\end{figure}
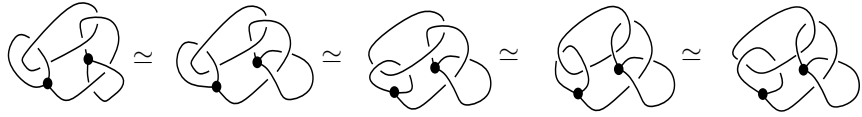
\end{proof}
 
\begin{lemma}\label{lm:chiral}
$5_1$, $6_3$, $6_6$, $6_7$, $6_8$, $6_{10}$ in Table \ref{tab:handlebodylinks} are chiral.
\end{lemma} 
\begin{proof}
Recall that, given a handlebody link $\HL$,
if $\HL$ and $r\HL$ are equivalent,
then there is an orientation-reversing self-homeomorphism
of $\sphere$ sending $\HL$ to $\HL$.

Observe that each of $5_1, 6_6, 6_8, 6_{10}$
admits an obvious decomposing annulus satisfying 
conditions in Theorem \ref{teo:uniqueness_decomposition_special_case};
particularly the annulus in each of them is unique. 
Their chirality then follows readily from 
the fact that torus links are chiral.

To see chirality of $6_3$, we observe that, given
a $(2,1)$-handlebody link $\HL$, any self-homeomorphism
of $\sphere$ preserving $\HL$ sends
the meridian $m$ and the preferred longitude $l$ 
of the circle component to $m^{\pm 1}$ and $l^{\pm 1}$,
respectively. In particular, any isomorphism on 
knot groups induced by such a homeomorphism
sends the conjugacy class of 
$m\cdot l $ in $\pi_1(\Compl{\HL})$ to  
the conjugacy class of  $m\cdot l$, $m^{-1}\cdot l$, 
$m\cdot l^{-1}$ or $m^{-1}\cdot l^{-1}$,
depending on whether the homeomorphism is orientation-preserving. 

Let $N$ be the number of conjugacy classes
of homomorphisms from $\pi_1(\Compl{\HL})$ to a finite group $\mathsf{G}$
that sends $m\cdot l$ (and hence $m^{-1} \cdot l^{-1} $)
 to $1$, and $rN$ 
the number of conjugacy classes
of homomorphisms from $\pi_1(\Compl{\HL})$ to $\mathsf{G}$
that sends $m\cdot l^{-1}$ (and hence $m^{-1}\cdot l$) to $1$.
Now, if $\HL$ and its mirror image $r\HL$ 
are equivalent, then $N=rN$. 
This is however not the case with $6_3$;
when $\mathsf{G}=A_5$, we have $(N,rN)=(77,111)$ as computed by \cite{appcontour}.

\begin{figure}[h]
\centering
\begin{subfigure}{.32\linewidth}
\captionsetup{font=footnotesize,labelfont=footnotesize}
\def\svgwidth{.8\columnwidth}
\begingroup%
  \makeatletter%
  \providecommand\color[2][]{%
    \errmessage{(Inkscape) Color is used for the text in Inkscape, but the package 'color.sty' is not loaded}%
    \renewcommand\color[2][]{}%
  }%
  \providecommand\transparent[1]{%
    \errmessage{(Inkscape) Transparency is used (non-zero) for the text in Inkscape, but the package 'transparent.sty' is not loaded}%
    \renewcommand\transparent[1]{}%
  }%
  \providecommand\rotatebox[2]{#2}%
  \newcommand*\fsize{\dimexpr\f@size pt\relax}%
  \newcommand*\lineheight[1]{\fontsize{\fsize}{#1\fsize}\selectfont}%
  \ifx\svgwidth\undefined%
    \setlength{\unitlength}{566.92913386bp}%
    \ifx\svgscale\undefined%
      \relax%
    \else%
      \setlength{\unitlength}{\unitlength * \real{\svgscale}}%
    \fi%
  \else%
    \setlength{\unitlength}{\svgwidth}%
  \fi%
  \global\let\svgwidth\undefined%
  \global\let\svgscale\undefined%
  \makeatother%
  \begin{picture}(1,0.9)%
    \lineheight{1}%
    \setlength\tabcolsep{0pt}%
    \put(0,0){\includegraphics[width=\unitlength,page=1]{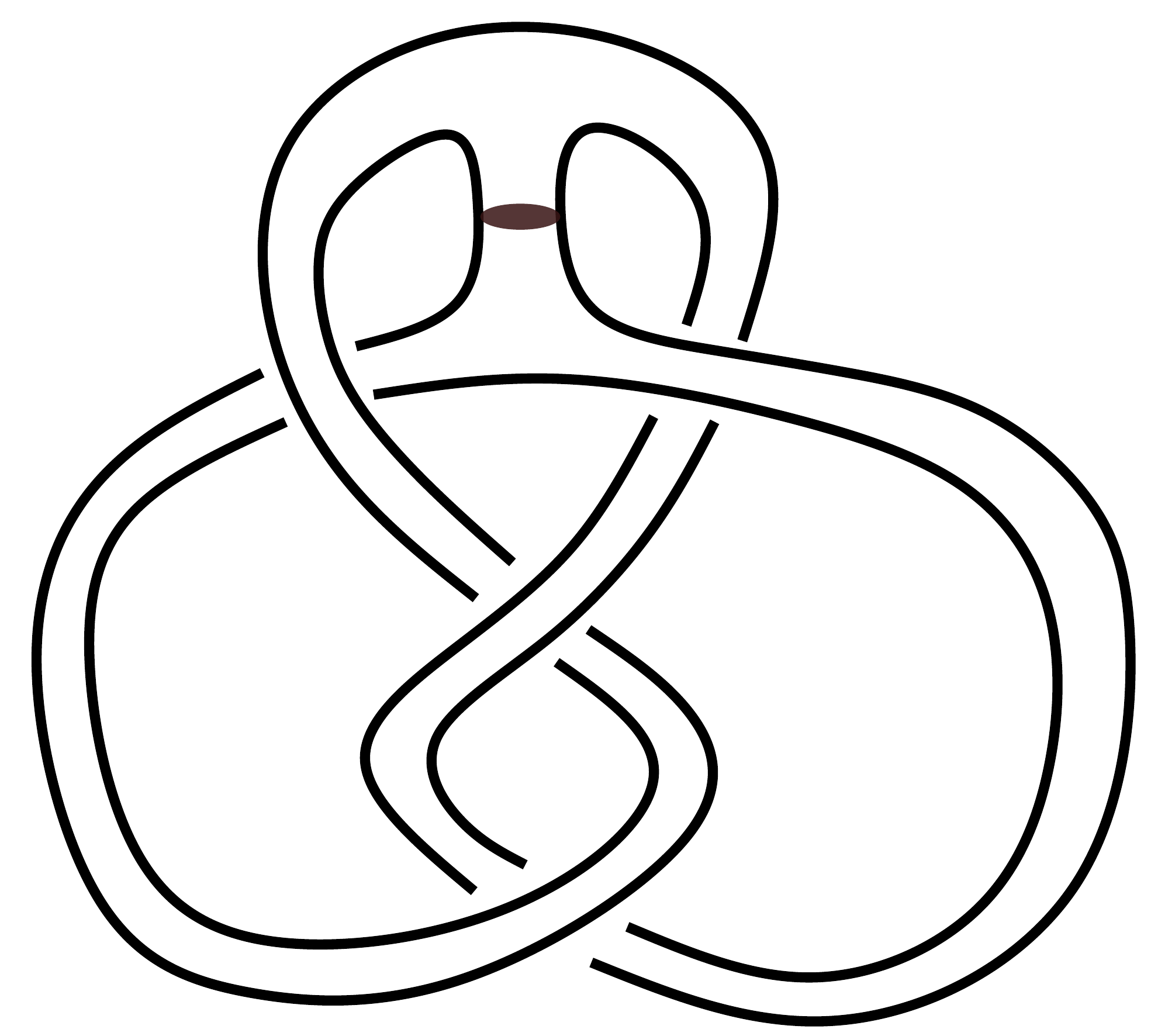}}%
    \put(0.33022681,0.69683137){\color[rgb]{0,0,0}\makebox(0,0)[lt]{\lineheight{1.25}\smash{\begin{tabular}[t]{l}{\tiny $D$}\end{tabular}}}}%
  \end{picture}%
\endgroup%
 
\caption{Pair $(T,D)$.}
\label{fig:handlebody-knot-disk}
\end{subfigure}
\begin{subfigure}{.32\linewidth}
\captionsetup{font=footnotesize,labelfont=footnotesize}
\def\svgwidth{.8\columnwidth}
\begingroup%
  \makeatletter%
  \providecommand\color[2][]{%
    \errmessage{(Inkscape) Color is used for the text in Inkscape, but the package 'color.sty' is not loaded}%
    \renewcommand\color[2][]{}%
  }%
  \providecommand\transparent[1]{%
    \errmessage{(Inkscape) Transparency is used (non-zero) for the text in Inkscape, but the package 'transparent.sty' is not loaded}%
    \renewcommand\transparent[1]{}%
  }%
  \providecommand\rotatebox[2]{#2}%
  \newcommand*\fsize{\dimexpr\f@size pt\relax}%
  \newcommand*\lineheight[1]{\fontsize{\fsize}{#1\fsize}\selectfont}%
  \ifx\svgwidth\undefined%
    \setlength{\unitlength}{566.92913386bp}%
    \ifx\svgscale\undefined%
      \relax%
    \else%
      \setlength{\unitlength}{\unitlength * \real{\svgscale}}%
    \fi%
  \else%
    \setlength{\unitlength}{\svgwidth}%
  \fi%
  \global\let\svgwidth\undefined%
  \global\let\svgscale\undefined%
  \makeatother%
  \begin{picture}(1,0.9)%
    \lineheight{1}%
    \setlength\tabcolsep{0pt}%
    \put(0,0){\includegraphics[width=\unitlength,page=1]{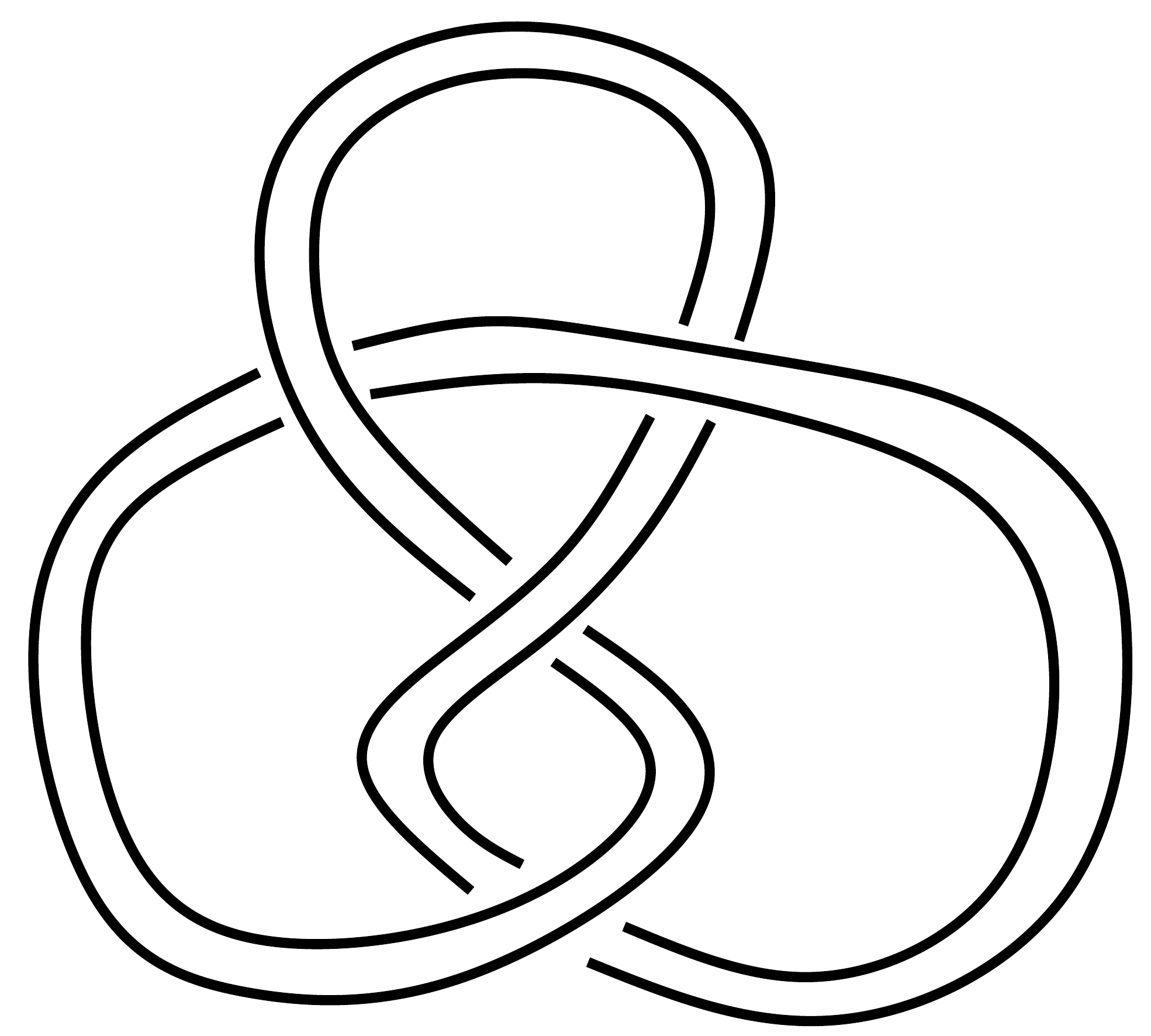}}%
    \put(0.45207966,0.68795992){\color[rgb]{0,0,0}\makebox(0,0)[lt]{\lineheight{1.25}\smash{\begin{tabular}[t]{l}{\footnotesize $\alpha$}\end{tabular}}}}%
    \put(0,0){\includegraphics[width=\unitlength,page=2]{solid-torus-knot-arc.pdf}}%
  \end{picture}%
\endgroup%
 
\caption{Dual pair $(\op{K4_1},\alpha)$.}
\label{fig:figure-eight-arc}
\end{subfigure}
\begin{subfigure}{.32\linewidth}
\captionsetup{font=footnotesize,labelfont=footnotesize}
\def\svgwidth{.8\columnwidth}
\begingroup%
  \makeatletter%
  \providecommand\color[2][]{%
    \errmessage{(Inkscape) Color is used for the text in Inkscape, but the package 'color.sty' is not loaded}%
    \renewcommand\color[2][]{}%
  }%
  \providecommand\transparent[1]{%
    \errmessage{(Inkscape) Transparency is used (non-zero) for the text in Inkscape, but the package 'transparent.sty' is not loaded}%
    \renewcommand\transparent[1]{}%
  }%
  \providecommand\rotatebox[2]{#2}%
  \newcommand*\fsize{\dimexpr\f@size pt\relax}%
  \newcommand*\lineheight[1]{\fontsize{\fsize}{#1\fsize}\selectfont}%
  \ifx\svgwidth\undefined%
    \setlength{\unitlength}{566.92913386bp}%
    \ifx\svgscale\undefined%
      \relax%
    \else%
      \setlength{\unitlength}{\unitlength * \real{\svgscale}}%
    \fi%
  \else%
    \setlength{\unitlength}{\svgwidth}%
  \fi%
  \global\let\svgwidth\undefined%
  \global\let\svgscale\undefined%
  \makeatother%
  \begin{picture}(1,0.9)%
    \lineheight{1}%
    \setlength\tabcolsep{0pt}%
    \put(0,0){\includegraphics[width=\unitlength,page=1]{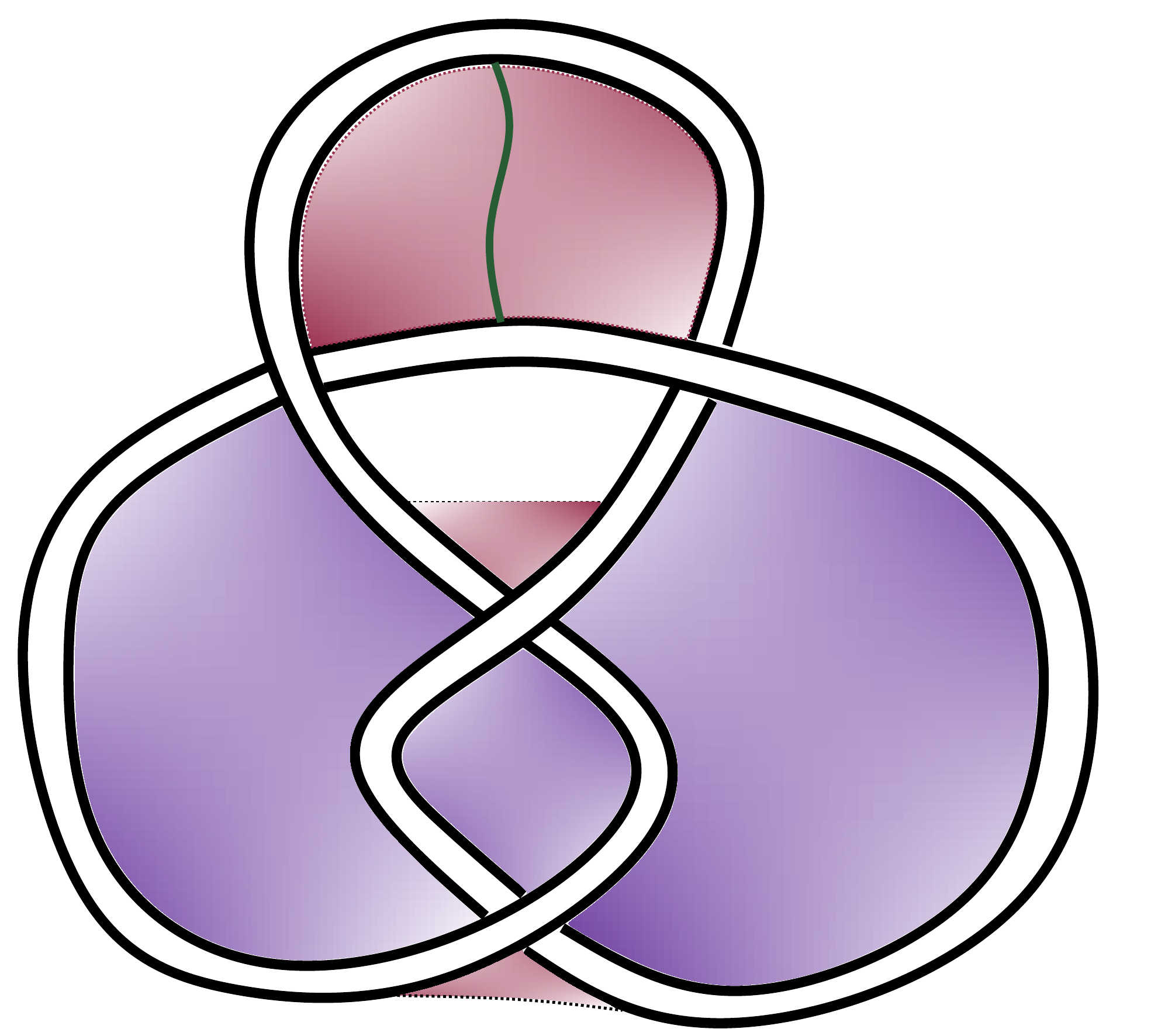}}%
    \put(0.71245232,0.29161377){\color[rgb]{0,0,0}\makebox(0,0)[lt]{\lineheight{1.25}\smash{\begin{tabular}[t]{l}$S$\end{tabular}}}}%
    \put(0.44142311,0.70673447){\color[rgb]{0,0,0}\makebox(0,0)[lt]{\lineheight{1.25}\smash{\begin{tabular}[t]{l}{\footnotesize $\alpha$}\end{tabular}}}}%
  \end{picture}%
\endgroup%
 
\caption{Seifert surface $S$. 
}
\label{fig:seifert_surface}
\end{subfigure}
\caption{•} 
\end{figure}
\begin{figure}[h]
\def\svgwidth{.5\columnwidth}
\begingroup%
  \makeatletter%
  \providecommand\color[2][]{%
    \errmessage{(Inkscape) Color is used for the text in Inkscape, but the package 'color.sty' is not loaded}%
    \renewcommand\color[2][]{}%
  }%
  \providecommand\transparent[1]{%
    \errmessage{(Inkscape) Transparency is used (non-zero) for the text in Inkscape, but the package 'transparent.sty' is not loaded}%
    \renewcommand\transparent[1]{}%
  }%
  \providecommand\rotatebox[2]{#2}%
  \newcommand*\fsize{\dimexpr\f@size pt\relax}%
  \newcommand*\lineheight[1]{\fontsize{\fsize}{#1\fsize}\selectfont}%
  \ifx\svgwidth\undefined%
    \setlength{\unitlength}{1417.32283465bp}%
    \ifx\svgscale\undefined%
      \relax%
    \else%
      \setlength{\unitlength}{\unitlength * \real{\svgscale}}%
    \fi%
  \else%
    \setlength{\unitlength}{\svgwidth}%
  \fi%
  \global\let\svgwidth\undefined%
  \global\let\svgscale\undefined%
  \makeatother%
  \begin{picture}(1,0.3)%
    \lineheight{1}%
    \setlength\tabcolsep{0pt}%
    \put(0,0){\includegraphics[width=\unitlength,page=1]{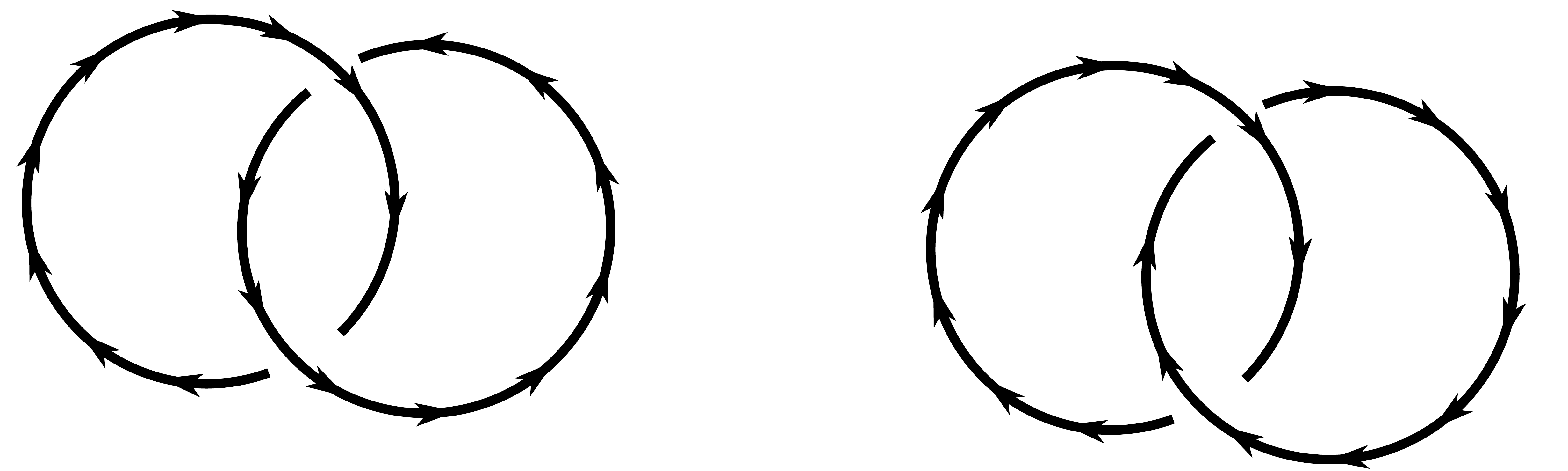}}%
    \put(-0.00015248,0.00373569){\color[rgb]{0,0,0}\makebox(0,0)[lt]{\lineheight{1.25}\smash{\begin{tabular}[t]{l}{\tiny link number $=-1$ }\end{tabular}}}}%
    \put(0.7104461,0.28224996){\color[rgb]{0,0,0}\makebox(0,0)[lt]{\lineheight{1.25}\smash{\begin{tabular}[t]{l}{\tiny link number $=1$}\end{tabular}}}}%
  \end{picture}%
\endgroup%
 
\caption{Hopf links with different orientations.}
\label{fig:two_hopf_links}
\end{figure} 
In the case of $6_7$
by Theorem \ref{teo:uniqueness_decomposition_special_case},
every self-homeomorphism $f$
of $\sphere$ sending $6_7$ to itself  
induces a self-homeomorphism $f$
sending the handlebody-knot-disk pair $(T,D)$ 
in Fig.\ \ref{fig:handlebody-knot-disk}
to itself, or alternatively the (fattened) 
figure eight with an arc
$(\op{K4_1},\alpha)$ in Fig.\ \ref{fig:figure-eight-arc} to itself, 
where $\alpha$ is the dual one-simplex to $D$.

Let $S$ be a minimal 
Seifert surface of the figure eight (Fig.\ \ref{fig:seifert_surface}) 
containing the arc $\alpha$.
Then it can be further assumed that $f(S)=S$, because
the complement of the tubular neighborhood $N(\alpha)$ of $\alpha$ in $S$
is a Seifert surface of a Hopf link, and up to ambient isotopy,
the Hopf link admits two minimal Seifert surfaces, and 
only one of them can give us $S$ after gluing $N(\alpha)$ back.
%

Since $f$ sends the complement $S\setminus N(\alpha)$
to $S\setminus f\big(N(\alpha)\big)$, 
if $f$ is orientation-reversing, then
the two oriented Hopf links in Fig.\ \ref{fig:two_hopf_links} 
are ambient isotopic,
but that is not possible by their linking numbers.
\end{proof}

\section{Reducible handlebody links}\label{sec:reducible}
In this section, we show that 
Table \ref{tab:reducible} classifies, up to ambient isotopy
and mirror image, all non-split, reducible 
$(n,1)$-handlebody links up to six crossings (Theorem \ref{teo:reducible}). 
We begin by considering the order-$1$ connected sum
for handlebody links.

\subsection{Order-1 connected sum}
A handlebody-link-component pair $(\HL,h)$ is 
a handlebody link $\HL$ with a selected component $h$ of $\HL$.  
\nada{
admit a diagram 
with \emph{edge-connectivity} one, and
the disconnecting arc joins two subgraphs that 
contain exactly one $3$-valent node each, which implies that
such disconnecting arc is unique.\draftYYY{In which sense it is unique? Do
we need to use uniqueness later in the works?}
\draftMMM{I intended to say that cannot exist two such disconnecting arcs in the
same diagram, the text could be changed perhaps in:}
we can have at most one disconnecting arc.

Reducible handlebody links 
can be constructed 
by the connected sum operation, also called $1$-sum.
}
\begin{definition}[\textbf{Order-$1$ connected sum}]\label{def:one_sum}
Let $(\HL_1,h_1)$ and $(\HL_2,h_2)$ be two handlebody-link-component pairs.
Then their order-$1$ connected sum $(\HL_1,h_1)$\onesum$(\HL_2,h_2)$ 
is given by removing the interior of
a $3$-ball $B_1$ (resp.\ $B_2$) in $\sphere$ 
with $B_1\cap \partial \HL_1= B_1\cap \partial h_1$ 
(resp. $B_2\cap \partial \HL_2= B_2\cap \partial h_2$) 
a $2$-disk,  
and then gluing the resulting $3$-manifolds $\Compl{B_1}$, $\Compl{B_2}$
via an orientation-reversing homeomorphism $f:(\partial B_1,(\partial B_1)\cap h_1)\rightarrow 
(\partial B_2,(\partial B_2)\cap h_2)$.  
We use $\HL_1$\onesum$\HL_2$ to denote 
the set of order-$1$ connected sums 
between $\HL_1,\HL_2$ with all possible selected components.
\end{definition} 
\begin{figure}[ht]
\def\svgwidth{0.8\columnwidth}
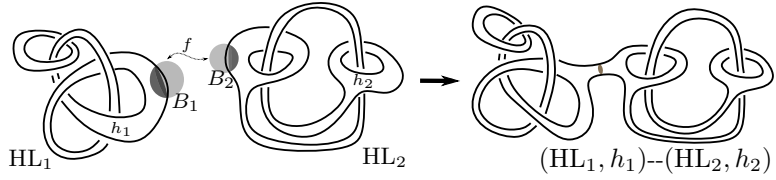   
\caption{Order-$1$ connected sum of $(\HL_1,h_1)$\onesum$(\HL_2,h_2)$.}
\label{fig:onesum}
\end{figure}
The following generalizes 
the case of handlebody knots in \cite[Theorem $2$]{Tsu:75}.
\begin{theorem}[\textbf{Uniqueness}]\label{teo:uniqueness_onesum}
Given a non-split, reducible $(n,1)$-handlebody link $\HL$,  
if $\HL\simeq (\HL_1,h_1)\text{\onesum} (\HL_2,h_2)$, and
$\HL\simeq (\HL_1',h_1')\text{\onesum} (\HL_2',h_2')$,
then $(\HL_i,h_i)\simeq (\HL_i',h_i')$, $i=1,2$, up to reordering. 
\end{theorem}
\begin{proof}
Note first that, since $\HL$ is non-split and reducible,
$\HL_i, \HL_i'$, $i=1,2$, are non-split, and  
$\pi_1(\Compl{\HL})$ is a non-trivial free product 
$\mathsf{G}_1\ast \mathsf{G}_2$,
where $\mathsf{G}_i$ is the knot group of $\HL_i$, $i=1,2$. 

Let $D$ and $D'$ be the separating disks
in $\Compl{\HL}$ given by
the factorizations 
$\HL\simeq (\HL_1,h_1)\text{\onesum} (\HL_2,h_2)$
and $\HL\simeq (\HL_1',h_1')\text{\onesum} (\HL_2',h_2')$,
respectively. 
Suppose neither $\mathsf{G}_1$ nor $\mathsf{G}_2$ 
is isomorphic to $\mathbb{Z}$.
Then, up to isotopy, $D'\cap D=\emptyset$ by the innermost circle/arc
argument.

Suppose one of $\mathsf{G}_i,i=1,2$, say $\mathsf{G}_1$, is
isomorphic to $\mathbb{Z}$, that is, 
$\HL_1$ is a trivial solid torus in $\sphere$. 
Then $\mathsf{G}_2$ must be non-cyclic, since $n>1$. 
Let $D_l$ be the disk bounded by the longitude of $\HL_1$, and 
isotopy $D,D_l$ 
such that the number $n$ (resp.\ $n_l$) 
of components of $D'\cap D$ (resp. $D'\cap D_l$)
is minimized. 

\noindent {\bf Claim: $n_l=0$.} 
Note first that the minimality implies that
$D'\cap D_l$ contains no circle components.
Now, consider a tubular neighborhood $N(D_l)$ of $D_l$ 
in $\Compl{\op{HL}}$ small enough 
such that $\overline{N(D_l)}\cap D=\emptyset$
and $\overline{N(D_l)}\cap D'$
are some disks, each of which intersects $D_l^+$ 
(resp.\ $D_l^-$) at exactly one arc on its boundary, where 
$D_l^{\pm}\subset \partial\overline{N(D_l)}$ are proper disks 
in $\Compl{\op{HL}}$ parallel to $D_l$.
The claim then follows once we have shown that $N(D_l)$ can 
be isotopied away from $D'$.

To see this, we construct a labeled tree $\Upsilon$ from
the complement of the intersection 
$D'\cap D_l^\pm$ in $D'$, where 
$D^\pm:=D_l^+\cup D_l^-$.
Regard each component of $D'\setminus 
\big(D'\cap D^\pm\big)$ as a node in $\Upsilon$, 
and each arc in $D' \cap D_l^\pm$ as
an edge in $\Upsilon$ connecting the two
nodes representing the components of  
$D'\setminus \big( D'\cap D^\pm\big)$
whose closures intersect at the arc.
Since each arc in $D'\cap  D^\pm$ 
cuts $D'$ into two, $\Upsilon$ is a tree.
The first two figures from the left in 
Fig.\ \ref{fig:asso_graph_cutting_disk}
illustrate the construction.
\begin{figure}[h]
\includegraphics[width=1.0\textwidth]{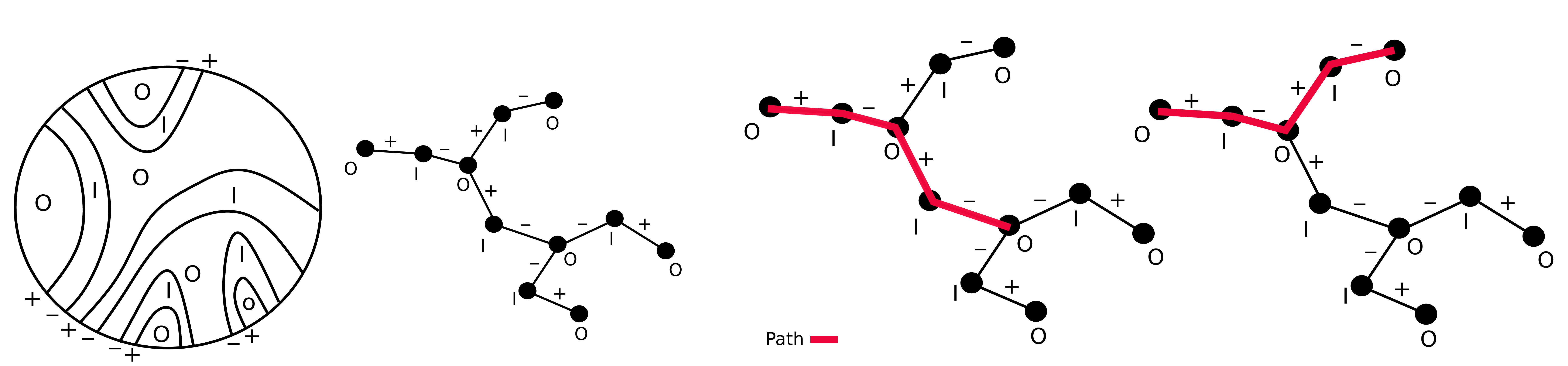} 
\caption{$h(D_l)$ and $\Upsilon$.}
\label{fig:asso_graph_cutting_disk}
\end{figure} 

We label nodes and edges of $\Upsilon$ as follows: 
A node is labeled with $I$ 
if the corresponding component of 
$D'\setminus \big(D'\cap D^\pm\big)$
is inside $N(D_l)$; otherwise the node is labeled with $O$. 
An edge of $\Upsilon$ is labeled with $+$
if the corresponding component of 
$D'\cap D^\pm_l$ is in $D_l^+$;
otherwise, it is labeled with $-$.

The labeling on $\Upsilon$ has the following
properties: (a) adjacent nodes have different labels;
(b) a node with label $I$ is bivalent, and the
two adjacent edges are labeled with $+$ and $-$,
respectively, whereas a node labeled with $O$
could be multi-valent; (c) a one-valent node
corresponds to an innermost arc in $D'$, and always has label $O$.

Consider a maximal path $\Gamma\subset \Upsilon$ 
starting from a one-valent node
and with the property that adjacent edges of $\Gamma$
have different labels. 
Then the other end point of the path must be labeled
with $O$ and it is either a one-valent node of $\Upsilon$
or a multi-valent node with all adjacent edges having the same label;
the two figures from the right in 
Fig.\ \ref{fig:asso_graph_cutting_disk} 
illustrate two possible maximal paths.

Without loss of generality, we may assume that
the adjacent edge of the starting one-valent node of $\Gamma$
is labeled with $+$. Denote the closure  
of the corresponding component of 
$ D' \setminus  \big(  D'\cap D^\pm\big)$
by $D_\Gamma^s$. Then $\partial D_\Gamma^s$ 
bounds a disk $T$ on $\partial (\HL \cup \overline{N(D_l)})$.
If $T\cap D_l^-=\emptyset$, then 
$D_\Gamma^s\cap D=\emptyset$ and hence $T\cap D=\emptyset$ by 
the minimality
of $n$; however, if it were the case, one could 
reduce $n_l$ 
by isotopying $D_l$ across the $3$-ball 
bounded by $D_\Gamma^s$ and $T$.
Hence $T$ must contain $D_l^-$.
Since the adjacent edge of the starting node is labeled with $+$, 
adjacent edges of the end node of $\Gamma$ in $\Upsilon$ 
are labeled with $-$. Denote by $D_\Gamma^e$
the closure of the component 
corresponding to the end node. Then $\partial D_\Gamma^e$
bounds a disk in  
$\partial (\HL \cup \overline{N(D_l)})$ 
that is contained in $T$ and 
has no intersection with $D_l^+$.
Particularly, $D_\Gamma^e\cap D=\emptyset$
by the minimality of $n$, and there is an arc in
$\partial D_\Gamma^e$
cutting a disk $D''$ 
off $\overline{T\setminus D_l^-}$ with 
$\mathring{D}''\cap D'=D''\cap D=\emptyset$,
so one can slide $N(D_l)$ over $D''$ (Fig.\ \ref{fig:sliding}) 
to decrease $n_l$, a contradiction.
\begin{figure}[h]
\def\svgwidth{0.8\columnwidth}
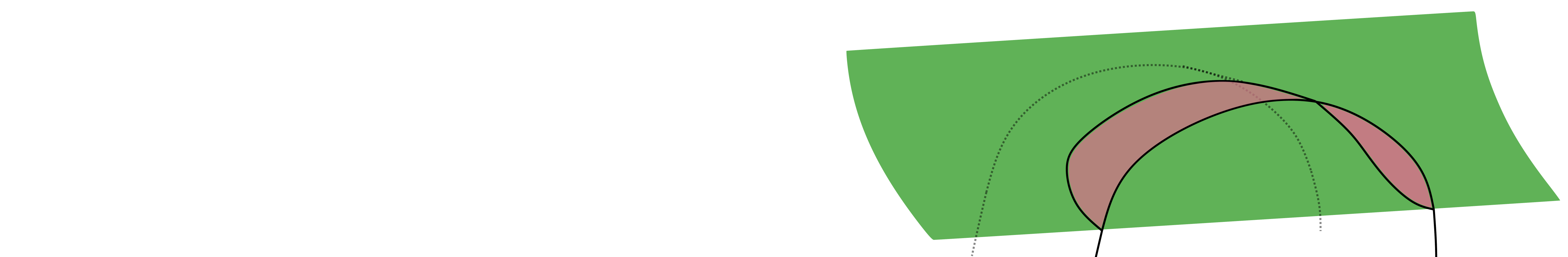   
\caption{}
\label{fig:sliding}
\end{figure}
Consequently, such a path $\Gamma$ cannot exist, but
this can happen only if $\Upsilon$ is empty.
The claim is thus proved. It implies that
$\HL_1,\HL_1'$ are trivial solid tori
in $\sphere$, and $\HL_2$, $\HL_2'$ are equivalent to
$\overline{N(D_l)}\cup \HL$.
\end{proof}
\subsection{Non-split, reducible handlebody links}
\begin{table}[h!]
  \begin{center}
    \caption{Reducible links with up to six crossings.}
    \label{tab:reducible}
    \begin{tabular}{c|c|l|c} 
      \textbf{crossings} & \textbf{$c(L_1)$ + $c(L_2)$} & \textbf{description} & \textbf{$\vert L_1$\onesum $L_2\vert$}\\
      \hline
      2 (1) & 0 + 2 & unknot \onesum\ Hopf & 1\\
      \hline
      4 (4) & 0 + 4 & unknot \onesum\ L4a1 & 1\\
            &       & unknot \onesum\ Hopf\#Hopf & 2 \\
            & 2 + 2 & Hopf \onesum\ Hopf & 1 \\
      \hline
      5 (4) & 0 + 5 & unknot \onesum\ Whitehead  & 1 \\
            &       & unknot \onesum\ Trefoil\#Hopf & 2 \\
            & 3 + 2 & trefoil \onesum\ Hopf & 1 \\
      \hline
      6 (17) & 0 + 6 & unknot \onesum\ L6a$i$, $i=1,\dots,5$  & 1 \\
            &       & unknot \onesum\ L6n1  & 1 \\
            &       & unknot \onesum\ L4a1\#Hopf  & 3 \\
            &       & unknot \onesum\ (Hopf\#Hopf)\#Hopf  & 4 \\
            & 2 + 4 & Hopf \onesum\ L4a1  & 1 \\
            &       & Hopf \onesum\ Hopf\#Hopf  & 2 \\
            & 4 + 2 & K4a1 \onesum\ Hopf  &  1 \\
      \hline
    \end{tabular}
  \end{center}
\end{table}
Table \ref{tab:reducible} lists all non-split,
reducible $(n,1)$-handlebody links obtained by performing 
order-$1$ connected sum on pairs of
links $(L_1,L_2)$ 
with crossing numbers $(c_1,c_2)$ and $c_1+c_2\leq 6$.
Since $n>1$, one of $L_1$, $L_2$, say $L_2$, is
a link with more than one component.
The number in parentheses indicates 
the total number of inequivalent reducible handlebody links 
of the given crossing number.
By Theorem \ref{teo:uniqueness_onesum}, isotopy types of 
$L_1$ and $L_2$ with selected components   
determine the isotopy type of the resulting handlebody link
$L_1$\onesum$L_2$.
Thus there are no duplicates in Table \ref{tab:reducible}.

On the other hand, 
by Lemmas \ref{lm:three_connected_IH_minimal}
and \ref{lm:two_connected_IH_minimal}
and Theorem \ref{teo:irreducibility},  
minimal diagrams of non-split, reducible $(n,1)$-handlebody links 
up to $6$ crossings cannot have $k$-connectivity, $k>1$.  
This shows the completeness of Table \ref{tab:reducible}. 
%
 
In particular, every non-split, reducible $(n,1)$-handlebody link
in Table \ref{tab:reducible}
admits a minimal diagram with $1$-connectivity; thus we postulate
the following conjecture.
\begin{conjecture}\label{conj:decomp_is_of_type_1}
Every non-split, reducible handlebody link admits a minimal
diagram with $1$-connectivity.  
\end{conjecture} 
Not every minimal diagram
of a reducible handlebody link has $1$-connectivity.
By Theorem \ref{teo:uniqueness_onesum},
Conjecture \ref{conj:decomp_is_of_type_1} 
implies the additivity of the crossing number (Conjecture \ref{conj:crossings_additivity}), 
a reminiscence of a one-hundred years old problem 
in knot theory.
\begin{conjecture}\label{conj:crossings_additivity} 
If $(\HL_1,h_1)\text{\onesum} (\HL_2,h_2)$ is a $(n,1)$-handlebody link, then 
\begin{equation}\label{eq:c_plus_c}
c\big((\HL_1,h_1)\text{\onesum} (\HL_2,h_2)\big)=c(\HL_1)+c(\HL_2).
\end{equation} 
\end{conjecture}


\section*{Acknowledgements}
The first author benefits from the support
of the GNA\-MPA (Gruppo Nazionale per l'Analisi Matematica, la Probabilit\`a
e le loro Applicazioni) of INdAM (Istituto Nazionale di Alta
Matematica). The second author
benefits from the support of the Swiss National Science Foundation Professorship grant PP00P2\_179110/1. 
The fourth author is supported by
National Center of Theoretical Sciences.

\appendix
\section{Output of the code}\label{sec:code}
\subsection{Minimal diagrams from the code}
The software code used in the paper exhaustively enumerates
$3$-edge-connected plane graphs with two trivalent vertices
and $q$ quadrivalent vertices, $0< q\leq 6$,
without double arcs that form a non-bigon.
Note that the trivial theta curve is the only $3$-edge-connected 
plane graph without quadrivalent vertices.   
The output of the code is examined and 
summarized in Table \ref{tab:numfromcode},
while the detailed list is available on 
\url{http://dmf.unicatt.it/paolini/handlebodylinks/},
where each plane graph is described by its adjacent matrix
together with a fixed ordering (clockwise or counterclockwise) 
of the edges adjacent to every vertex, as determined by the planar embedding. 
%
%
%
%
\subsubsection{Four crossings or less}
In Table \ref{tab:detailsfor1_4}, we analyze
the output of the code up to $4$ quadrivalent vertices,
where the column ``quad.\ v.'' lists the number of quadrivalent vertices
and ``ref.\ no.'' the reference number
of each plane graph in the output of the code. 
The column ``induced diagrams'' describes 
minimality of diagrams induced by each plane graph.
Most induced diagrams  
are not minimal, and we record those that are and their
isotopy types as special graphs or handlebody links, up to mirror image. 
Up to $4$-crossings, no IH-minimal diagram with more than one 
component is found.
%
%
%
\begin{table}[ht!]
  \begin{center}
    \caption{Diagrams with up to $4$ crossings.}
    \label{tab:detailsfor1_4}
    \begin{tabular}{|c|c|l|} 
    \hline
    \textbf{quad.\ v.}&\textbf{ref. no.} & \textbf{induced diagrams} \\
    \hline
    1 & none &  none \\
    \hline
    2 & \#1 & R-minimal; G$2_1$ in Table \ref{tab:graphs}; not IH-minimal\\
    \hline
    \multirow{2}{*}{3}& \#1, & not R-minimal \\
    \cline{2-3}
    & \#2,\#3 & R-minimal; G$3_2$ in Table \ref{tab:graphs}; not IH-minimal\\
    \hline
    \multirow{6}{*}{4}& \#1,\#2 & IH-minimal; G$4_1$ in Table \ref{tab:graphs}\\
    \cline{2-3}
    & \#3 & R-minimal; G$3_2$ in Table \ref{tab:graphs}; not IH-minimal\\
    \cline{2-3}
    & \#4, \#8 & R-minimal; G$4_2$ in Table \ref{tab:graphs}; not IH-minimal\\
    \cline{2-3}
    & \#5, \#6, \#7 & R-minimal; G$4_3$ in Table \ref{tab:graphs}; not IH-minimal\\ 
    \cline{2-3}
    & \#9 & R-minimal; G$4_4$ in Table \ref{tab:graphs}; not IH-minimal\\
    \cline{2-3}
    & \#10 & R-minimal; G$4_5$ in Table \ref{tab:graphs}; not IH-minimal\\
    \hline
    \end{tabular}
  \end{center}
\end{table}

\subsubsection{Five and six crossing cases}
\begin{table}[ht!]
  \centering
    \caption{Diagrams with $5$ crossings.}
    \label{tab:detailsfor5}
    \begin{tabular}{|c|l|} 
    \hline
    \textbf{ref.\ no.} & \textbf{description} \\
    \hline
    \#6, \#11, \#14 & not R-minimal \\
    \hline
    \#22, \#26, \#35 & not IH-minimal \\
    \hline
    \#36 & not IH-minimal \\
    \hline
    \#37 & not R-minimal \\
    \hline
    \end{tabular}
   
\end{table}
In the $5$ crossings case, the code finds
$8$ plane graphs with more than one components, out of a total of $37$ 
planar embeddings.
Table \ref{tab:detailsfor5} records the analysis for their induced diagrams; none of them gives IH-minimal diagrams.
In the $6$ crossing case, out of $181$ plane graphs, $37$ induces
diagrams with more than one components.
Table \ref{tab:detailsfor6} records the minimality of their induced diagrams. 
\begin{table}[ht!]
\centering   
    \caption{Diagrams with $6$ crossings.}
    \label{tab:detailsfor6}
    \begin{tabular}{|c|l|} 
    \hline
    \textbf{ref.\ no.} & \textbf{description} \\
    \hline
    \#5 & 
    \draftYYY{Reidemeister-equivalent to \#161}
    $6_1$ in Table \ref{tab:handlebodylinks} \\
    \hline
    \#15, \#22, \#34, \#45, \#54 & not R-minimal \\
    \hline
    \#56 & 
    \draftYYY{Reidemeister-equivalent to \#84}  
    $6_1$ in Table \ref{tab:handlebodylinks}\\
    \hline
    \#60 & $6_2$ in Table \ref{tab:handlebodylinks} \\
    \hline
    \#70 & $6_3$ in Table \ref{tab:handlebodylinks}  \\
    \hline
    \#73 & not R-minimal \\
    \hline
    \#83 & 
    \draftYYY{Reidemeister-equivalent to \#60} 
    $6_2$ in Table \ref{tab:handlebodylinks}\\
    \hline
    \#84 & $6_1$ in Table \ref{tab:handlebodylinks}\\
    \hline
    \#86, \#91, \#92, \#93 & not R-minimal \\
    \hline
    \#104, \#105, \#114, \#117, \#123 & not IH-minimal \\
    \hline
    \#134, \#135, \#137, \#144 & not IH-minimal \\
    \hline
    \#161, \#165 & 
    \draftYYY{IH-equivalent to \#84} 
    $6_1$ in Table \ref{tab:handlebodylinks}\\
    \hline
    \#168, \#169, \#170,\#171 & not IH-minimal \\
    \hline
    \#175 & 
    \draftYYY{IH equivalent to \#181} 
    $6_9$ in Table \ref{tab:handlebodylinks}\\
    \hline
    \#176 & not IH-minimal \\
    \hline
    \#177 & not R-minimal \\
    \hline
    \#179, \#180  & not IH-minimal \\
    \hline
    \#181 & $6_9$ in Table \ref{tab:handlebodylinks}\\
    \hline
    \end{tabular}
\end{table}



\draftMMM{I guess that this section can be completely omitted from the paper.
Or alternatively partially absorbed in the text somewhere else.}
\draftYYY{Let's keep it at least for the arXiv version}
Fig.\ \ref{fig:examples_from_code} exemplifies how the analysis
is done. Fig.\ \ref{fig:graph_2_6_n5} shows
how the diagrams induced by Plane Graph \#5 
are equivalent to
those by \#161 and \#165 in the case of 6 crossings,
and Fig.\ \ref{fig:reduction_6_n168} explains non-minimality
of diagrams induced by
Plane Graphs \#168, \#169, \#170, \#171. 
\begin{figure}[ht]
\centering
\begin{subfigure}{0.53\linewidth}
\includegraphics[width=0.95\textwidth]{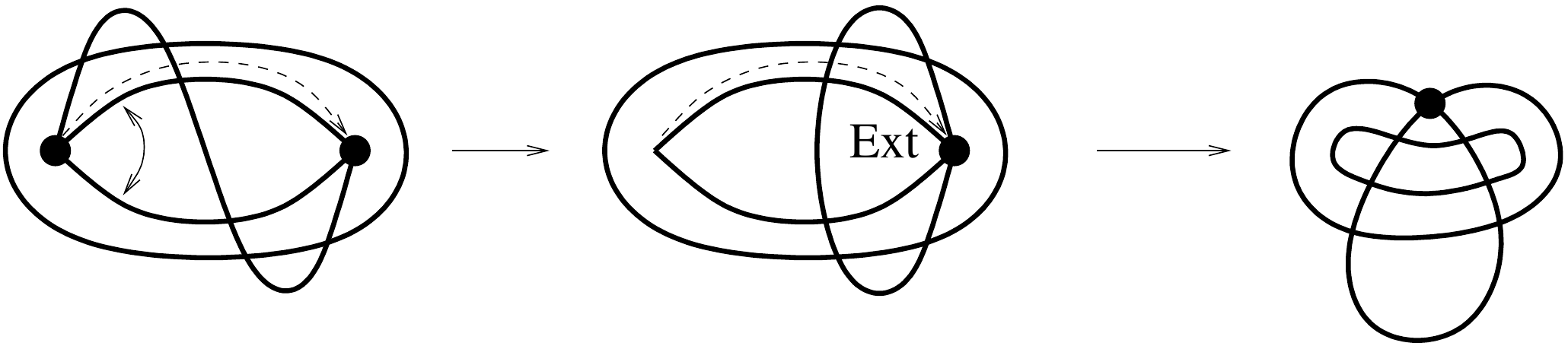}
\caption{Equivalent handlebody links from plane graphs.}
\label{fig:graph_2_6_n5}
\end{subfigure}
\begin{subfigure}{0.45\linewidth}
\includegraphics[width=0.95\textwidth]{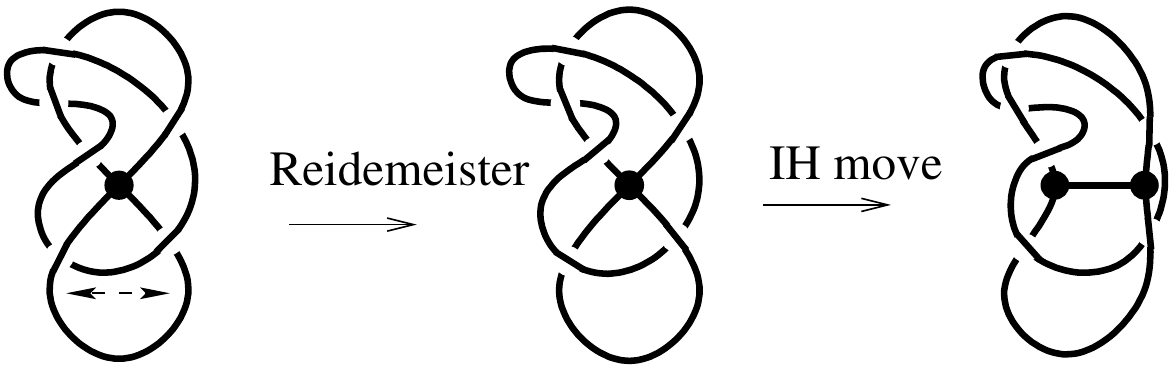}
\caption{Non-minimal diagrams.}
\label{fig:reduction_6_n168}
\end{subfigure}
\caption{•}
\label{fig:examples_from_code}
\end{figure}
\subsubsection{Inequivalent planar embeddings}
As a side remark, Fig.\ \ref{fig:planar_embeddings} illustrates 
two examples of abstract graphs with inequivant planar embeddings: one with five quadrivalent vertices 
and the other with six. 
Note that the abstract graphs 
have $2$-vertex-connectivity, consistent with
the Whitney uniqueness theorem \cite{Whi:32}.
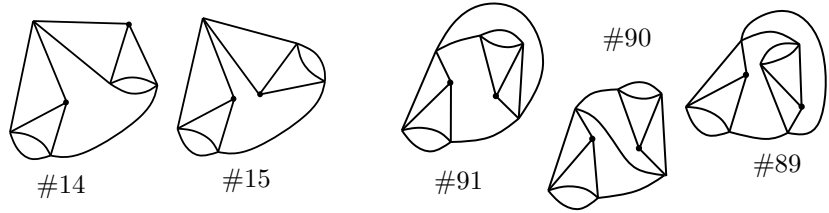
\begin{figure}[ht]
\centering
\begin{subfigure}{0.4\linewidth}
 \def\svgwidth{.85\columnwidth}
\begingroup%
  \makeatletter%
  \providecommand\color[2][]{%
    \errmessage{(Inkscape) Color is used for the text in Inkscape, but the package 'color.sty' is not loaded}%
    \renewcommand\color[2][]{}%
  }%
  \providecommand\transparent[1]{%
    \errmessage{(Inkscape) Transparency is used (non-zero) for the text in Inkscape, but the package 'transparent.sty' is not loaded}%
    \renewcommand\transparent[1]{}%
  }%
  \providecommand\rotatebox[2]{#2}%
  \newcommand*\fsize{\dimexpr\f@size pt\relax}%
  \newcommand*\lineheight[1]{\fontsize{\fsize}{#1\fsize}\selectfont}%
  \ifx\svgwidth\undefined%
    \setlength{\unitlength}{850.39370079bp}%
    \ifx\svgscale\undefined%
      \relax%
    \else%
      \setlength{\unitlength}{\unitlength * \real{\svgscale}}%
    \fi%
  \else%
    \setlength{\unitlength}{\svgwidth}%
  \fi%
  \global\let\svgwidth\undefined%
  \global\let\svgscale\undefined%
  \makeatother%
  \begin{picture}(1,0.6)%
    \lineheight{1}%
    \setlength\tabcolsep{0pt}%
    \put(0,0){\includegraphics[width=\unitlength,page=1]{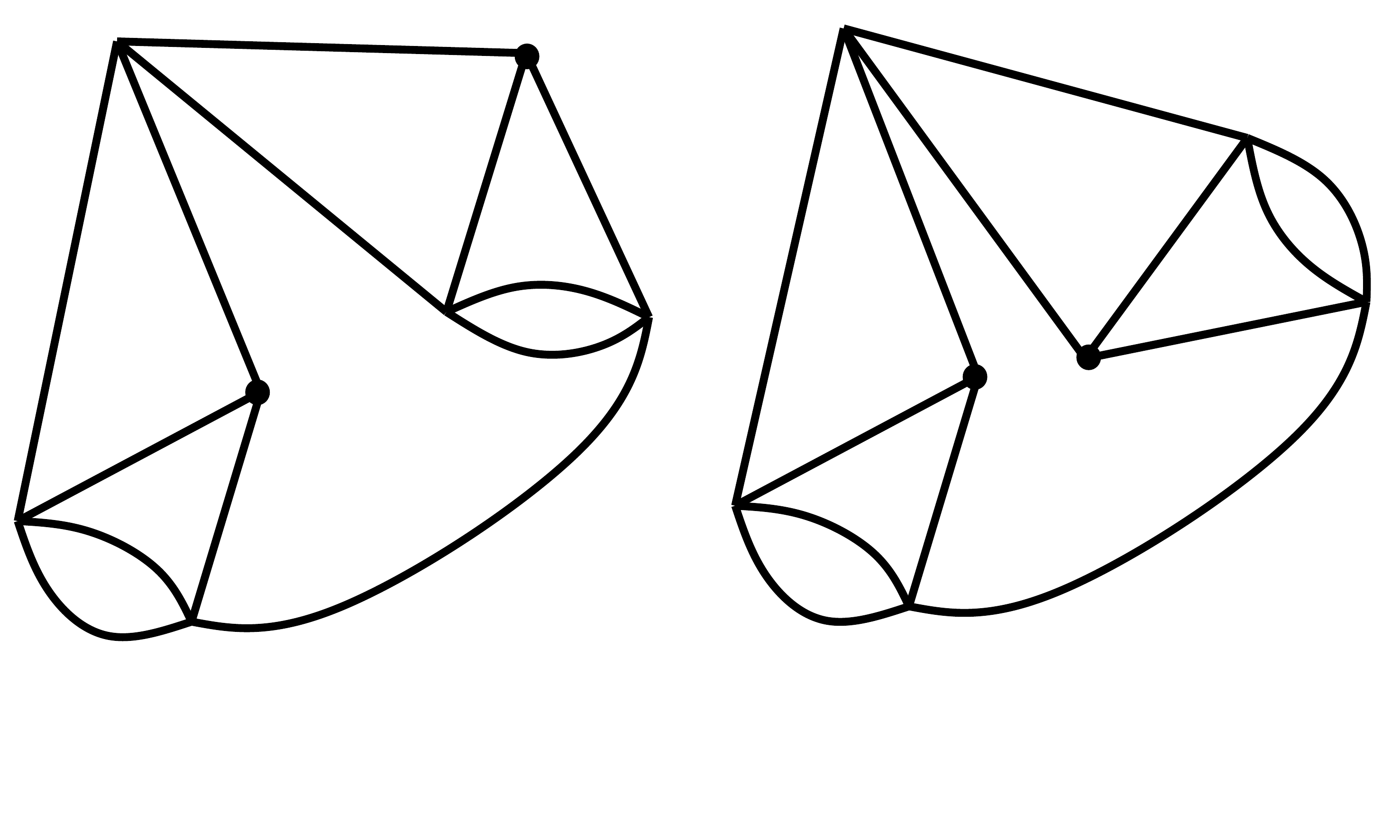}}%
    \put(0.09377383,0.04116903){\color[rgb]{0,0,0}\makebox(0,0)[lt]{\lineheight{1.25}\smash{\begin{tabular}[t]{l}$\# 14$\end{tabular}}}}%
    \put(0.66099111,0.05794159){\color[rgb]{0,0,0}\makebox(0,0)[lt]{\lineheight{1.25}\smash{\begin{tabular}[t]{l}$\# 15$\end{tabular}}}}%
  \end{picture}%
\endgroup%

\caption{Five-quadrivalent-vertex graph.}
\label{fig:planar_embeddings_5}
\end{subfigure}
\begin{subfigure}{0.5\linewidth}
\def\svgwidth{.9\columnwidth}
\begingroup%
  \makeatletter%
  \providecommand\color[2][]{%
    \errmessage{(Inkscape) Color is used for the text in Inkscape, but the package 'color.sty' is not loaded}%
    \renewcommand\color[2][]{}%
  }%
  \providecommand\transparent[1]{%
    \errmessage{(Inkscape) Transparency is used (non-zero) for the text in Inkscape, but the package 'transparent.sty' is not loaded}%
    \renewcommand\transparent[1]{}%
  }%
  \providecommand\rotatebox[2]{#2}%
  \newcommand*\fsize{\dimexpr\f@size pt\relax}%
  \newcommand*\lineheight[1]{\fontsize{\fsize}{#1\fsize}\selectfont}%
  \ifx\svgwidth\undefined%
    \setlength{\unitlength}{1048.81889764bp}%
    \ifx\svgscale\undefined%
      \relax%
    \else%
      \setlength{\unitlength}{\unitlength * \real{\svgscale}}%
    \fi%
  \else%
    \setlength{\unitlength}{\svgwidth}%
  \fi%
  \global\let\svgwidth\undefined%
  \global\let\svgscale\undefined%
  \makeatother%
  \begin{picture}(1,0.51351351)%
    \lineheight{1}%
    \setlength\tabcolsep{0pt}%
    \put(0,0){\includegraphics[width=\unitlength,page=1]{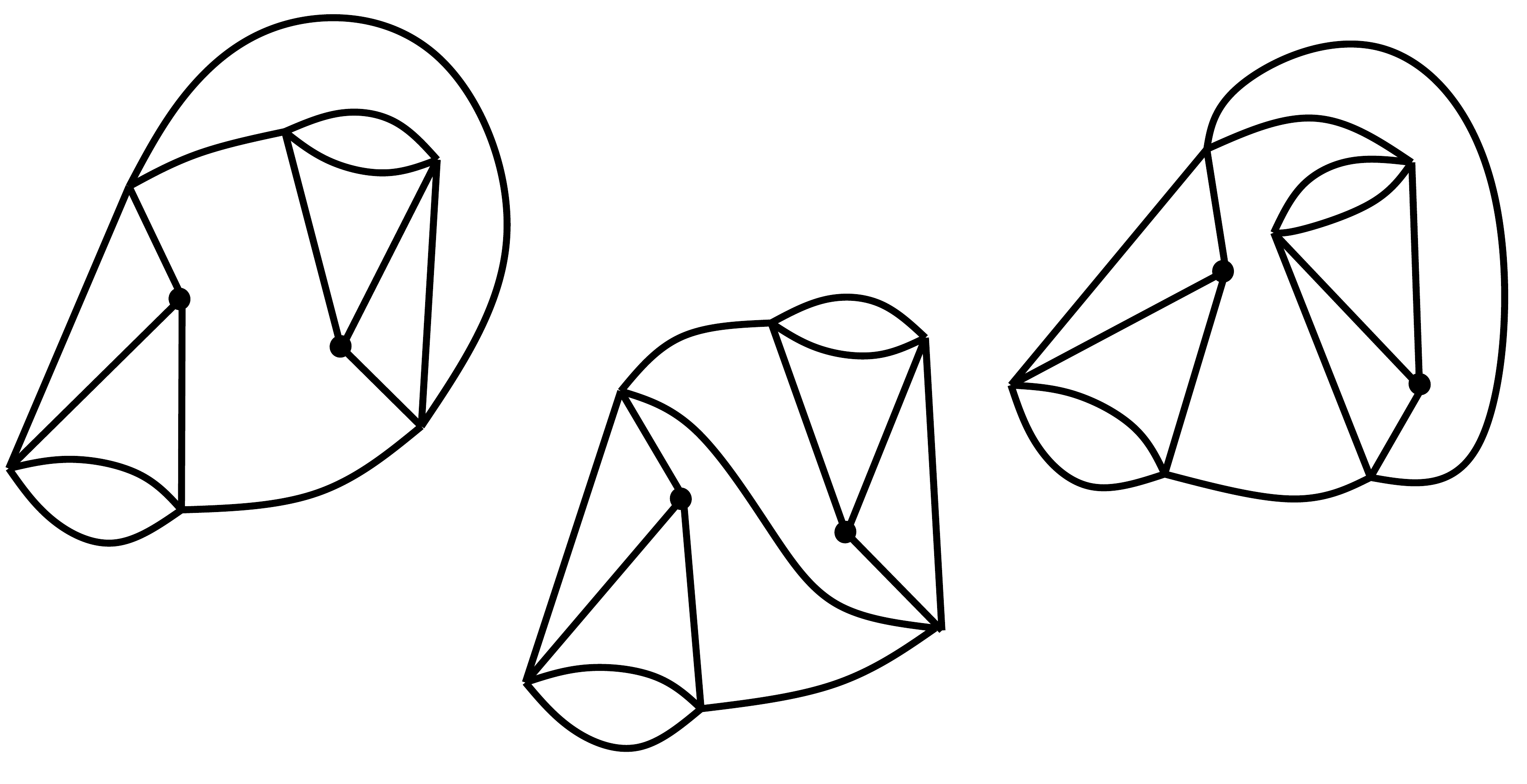}}%
    \put(0.08053221,0.06932777){\color[rgb]{0,0,0}\makebox(0,0)[lt]{\lineheight{1.25}\smash{\begin{tabular}[t]{l}$\#91$\end{tabular}}}}%
    \put(0.46918764,0.40336135){\color[rgb]{0,0,0}\makebox(0,0)[lt]{\lineheight{1.25}\smash{\begin{tabular}[t]{l}$\#90$\end{tabular}}}}%
    \put(0.81652667,0.10854344){\color[rgb]{0,0,0}\makebox(0,0)[lt]{\lineheight{1.25}\smash{\begin{tabular}[t]{l}$\#89$\end{tabular}}}}%
  \end{picture}%
\endgroup%
  
\caption{Six-quadrivalent-vertex graph.}
\label{fig:planar_embeddings_6}
\end{subfigure}
\caption{Inequivalent planar embeddings.}
\label{fig:planar_embeddings}
\end{figure}

\end{document}